\documentclass[english]{amsart}

\usepackage{amsmath,amssymb,enumerate}
\usepackage{amscd,amsxtra,calc}

\usepackage[T1]{fontenc}
\usepackage[utf8]{inputenc}
\usepackage[all]{xy}

\usepackage{babel}
\usepackage{amstext}
\usepackage{amsmath}
\usepackage{amsfonts}
\usepackage{latexsym}
\usepackage{ifthen}
\usepackage{verbatim}
\usepackage{enumerate}
\usepackage{tikz}\usetikzlibrary{decorations.markings,matrix,arrows}
\usepackage{tikz-cd}

\xyoption{all}
\newtheorem{lemma1}{}[section]

\newenvironment{lemma}{\begin{lemma1}{\bf Lemma.}}{\end{lemma1}}
\newenvironment{example}{\begin{lemma1}{\bf Example.}\rm}{\end{lemma1}}

\newenvironment{theorem}{\begin{lemma1}{\bf Theorem.}}{\end{lemma1}}
\newenvironment{proposition}{\begin{lemma1}{\bf Proposition.}}{\end{lemma1}}
\newenvironment{corollary}{\begin{lemma1}{\bf Corollary.}}{\end{lemma1}}
\newenvironment{remark}{\begin{lemma1}{\bf Remark.}\rm}{\end{lemma1}}
\newenvironment{definition}{\begin{lemma1}{\bf Definition.}}{\end{lemma1}}

\newenvironment{construction}{\begin{lemma1}{\bf Construction.}}{\end{lemma1}}
\newenvironment{conjecture}{\begin {lemma1}{\bf Conjecture.}}{\end{lemma1}}

\newenvironment{assumption}{\begin{lemma1}{\bf Assumption.}\rm}{\end{lemma1}}

\newenvironment{remark*}{{\bf Remark.}}{}
\newenvironment{example*}{{\bf Example.}}{}

\newcommand{\R}{\ensuremath{\mathbb{R}}}
\newcommand{\Q}{\ensuremath{\mathbb{Q}}}
\newcommand{\Z}{\ensuremath{\mathbb{Z}}}
\newcommand{\C}{\ensuremath{\mathbb{C}}}
\newcommand{\N}{\ensuremath{\mathbb{N}}}
\newcommand{\PP}{\ensuremath{\mathbb{P}}}

\newcommand{\W}{\ensuremath{\mathbb{W}}}

\newcommand{\set}[1]{\left\{#1\right\}}

\newcommand{\merom}[3]{\ensuremath{#1:#2 \dashrightarrow #3}}

\newcommand{\holom}[3]{\ensuremath{#1:#2  \rightarrow #3}}
\newcommand{\fibre}[2]{\ensuremath{#1^{-1} (#2)}}

\makeatletter
\ifnum\@ptsize=0  \fi \ifnum\@ptsize=2  \fi 

\newcommand\sL{{\mathcal L}}
\newcommand\sO{{\mathcal O}}

\newcommand\sX{{\mathcal X}}

\newcommand\bR{{\mathbb R}}
\newcommand\bZ{{\mathbb Z}}
\newcommand\bC{{\mathbb C}}

\newcommand\sW{{\mathcal W}}

\newcommand{\Gal}[0]{\operatorname{Gal}}

\newcommand{\To}{\longrightarrow}

\newcommand{\wti}{\widetilde}

\newcommand{\bk}{\mathbf{k}}
\newcommand{\bm}{\mathbf{m}}

\newcommand{\ga}{\alpha}
\newcommand{\gS}{\Sigma}

\newcommand{\gD}{\Delta}
\newcommand{\gk}{\kappa}

\newcommand{\go}{\omega}

\newcommand{\cH}{\mathcal{H}}

\newcommand{\cJ}{\mathcal{J}}
\newcommand{\cV}{\mathcal{V}}

\newcommand{\cQ}{\mathcal{Q}}
\newcommand{\cL}{\mathcal{L}}

\newcommand{\cX}{\mathcal{X}}

\newcommand{\cO}{\mathcal{O}}

\newcommand{\colonec}{\mathrel{:=}}
\newcommand{\bss}{\backslash}

\newcommand{\Id}{\mathrm{Id}}

\newcommand{\pr}{\mathrm{pr}}

\newcommand{\dto}{\dashrightarrow}
\newcommand{\hto}{\hookrightarrow}

\title{The fundamental group of compact K\"ahler threefolds}
\date{\today}

\author{Beno\^it Claudon}
\author{Andreas H\"oring}
\author{Hsueh-Yung Lin}

\address{Beno\^it Claudon, Institut \'Elie Cartan de Lorraine, Universit\'e de Lorraine, B.P. 70239, 54506 Vand{\oe}uvre-l\`es-Nancy Cedex, France}
\email{Benoit.Claudon@univ-lorraine.fr}

\address{Andreas H\"oring, Universit\'e C\^ote d'Azur, CNRS, LJAD, France}
\email{Andreas.Hoering@unice.fr}

\address{Hsueh-Yung Lin}
\email{hsuehyung.lin.math@gmail.com}

\begin{document}

\begin{abstract} 
Let $X$ be a compact K\"ahler manifold of dimension three. We prove that there exists a projective manifold $Y$ such that $\pi_1(X) \simeq \pi_1(Y)$. We also prove the bimeromorphic existence of algebraic approximations for compact K\"ahler manifolds
of algebraic dimension $\dim X-1$. Together with the work of Graf and the third author, this settles in particular the bimeromorphic Kodaira problem for
compact K\"ahler threefolds. 
\end{abstract}

\maketitle

\section{Introduction}

\subsection{Main result}

Compact K\"ahler manifolds arise naturally as generalisations of complex projective manifolds,
and Kodaira's problem asked if every compact K\"ahler manifold is deformation equivalent 
to a projective manifold. A positive answer to this problem trivially implies that the larger
class of K\"ahler manifolds realises the same topological invariants. The classification of analytic surfaces \cite{Kod} implies a positive answer to Kodaira's problem in this case (cf. also \cite{Buc08} for a different approach). However, Voisin's counterexamples \cite{Voi04, Voi07} show that there exist compact K\"ahler manifolds
of dimension at least four that do not deform to projective ones. Nevertheless it is interesting
to study Kodaira's problem at the level of some specific topological invariants, like the fundamental group:

\begin{conjecture} \label{conjecturegroups}
Let $X$ be a compact K\"ahler manifold.  Then the fundamental group $\pi_1(X)$ is projective, i.e. there exists a projective manifold $M$ such that $\pi_1(X) \simeq \pi_1(M)$.
\end{conjecture}

Note that unlike other problems on fundamental groups, this conjecture does not reduce to the case of surfaces: while, by the Lefschetz hyperplane theorem, the fundamental group of any projective manifold is realised by a projective surface, it is a priori not clear if the same holds in the K\"ahler category. Several partial results on Conjecture \ref{conjecturegroups} have
been obtained in the last years (cf.~\cite{CCE1},~\cite[Th\'eor\`eme 0.2]{CCE2}, and~\cite[Corollary 1.3]{Cla16}). In this paper we give a complete answer in dimension three:

\begin{theorem} \label{theorempione}
Let $X$ be a smooth compact K\"ahler threefold. Then $\pi_1(X)$ is  projective.
\end{theorem}

The proof of the result comes in several steps: if $X$ is covered by rational curves,
then its MRC-fibration $X \dashrightarrow Z$ induces an isomorphism 
$\pi_1(X) \simeq \pi_1(Z)$ \cite[Theorem 5.2]{Kol93}
\cite[Corollary 1]{BC15}, so we are done. If $X$ is not covered by rational curves,
we make a case distinction based on the algebraic dimension, i.e. the transcendence degree of the field of meromorphic functions on $X$. The case $a(X)=0$ has been solved in \cite{CC14} (see also \cite[Corollary 1.8]{Gra16}), and for the case $a(X)=1$ we can describe in detail the structure of the fundamental group using the algebraic reduction $X \dto C$ 
dominating a curve\footnote{After the submission of the first
version of this paper to the arxiv, the third author posted his preprint \cite{Lin16} on algebraic approximation which implies this case. Our proof is completely different and should be useful for generalisations to higher dimension.}. 
The most difficult case is when $a(X)=2$, where the resolution of the algebraic reduction defines
an elliptic fibration $X' \rightarrow S$ over a surface. Here the structure of the fundamental group is not known, even for a projective threefold. The main contribution of this paper is to use the theory of elliptic fibrations developed by Nakayama~\cite{NakW, Nak02c, NakLocal} to show the existence of a smooth bimeromorphic model $X'$ of $X$ which admits algebraic approximations. The arguments work in fact without the assumption that $\dim X = 3$. We will then derive the projectivity of K\"ahler groups in this case as a corollary.

\subsection{Algebraic approximation} 

Voisin's examples \cite{Voi07} 
show that there exist compact K\"ahler manifolds of dimension at least eight
such that none of their {\em smooth} bimeromorphic models deform to a projective manifold. Her examples are uniruled and up to now, no non-uniruled manifold has been discovered which satisfies the same property. In higher dimension, mild singularities occur naturally in the bimeromorphic models considered by the minimal model program. In this spirit, Peternell and independently Campana proposed a more flexible, bimeromorphic version of Kodaira's problem: 

\begin{conjecture} \label{conjectureapproximation}
Let $X$ be a compact K\"ahler manifold that is not uniruled. Then there exists a bimeromorphic map $X \dashrightarrow X'$ to a normal compact K\"ahler space $X'$ with terminal singularities that admits an algebraic approximation (cf. Definition \ref{definitionapproximation}).
\end{conjecture}

Algebraic approximation provides an explicit way to prove that a K\"ahler group
is projective: since $X'$ has terminal singularities, the fundamental group is invariant under the bimeromorphic map $X \dashrightarrow X'$ \cite{Tak03}. 
If we can always choose the algebraic approximation $\mathcal X \rightarrow \Delta$ to be a locally trivial deformation in the sense of \cite[p.627]{FK87} \cite{Ser06}, then Conjecture
\ref{conjectureapproximation} implies Conjecture \ref{conjecturegroups}.

Very recently Graf \cite{Gra16} and the third author \cite{Lin16, Lin17a, Lin17b} have made progress on the Kodaira problem,
by proving the existence of algebraic approximations for all smooth compact K\"ahler threefolds of Kodaira dimension $\gk$ at most one (including the case $\gk = -\infty$, namely uniruled threefolds~\cite[Cor.1.4]{a21}).
We prove Conjecture \ref{conjectureapproximation} for manifolds of the highest algebraic dimension: 

\begin{theorem} \label{theoremapproximation}
Let $X$ be a compact K\"ahler manifold of algebraic dimension $a(X)= \dim X - 1$.
Then there exists a bimeromorphic map $X' \dashrightarrow X$ 
such that $X'$ is a compact K\"ahler manifold admitting an algebraic approximation.
\end{theorem}

As $\kappa(X) \le a(X)$, the conclusion of Theorem~\ref{theoremapproximation} holds in particular when $\kappa(X) = \dim X-1$. Thus Theorem \ref{theoremapproximation} together with \cite{Gra16, Lin16, Lin17b} establish Conjecture \ref{conjectureapproximation} for all compact K\"ahler threefolds. As far as we can see, our techniques do not imply the existence of algebraic approximations for threefolds whose algebraic reduction $X' \rightarrow S$ is over a surface. In fact while resolving the bimeromorphic map $X' \dashrightarrow X$ appearing in our statement, one might blow up some curves that 
are not contracted by $X' \rightarrow S$. In this case it is difficult to relate the deformation theories
of $X$ and $X' \rightarrow S$. Thus the biregular Kodaira problem is still open for threefolds with $a(X) = 2$.

Finally as we already mentioned above, Theorem~\ref{theoremapproximation} has the following immediate corollary on the projectivity of K\"ahler fundamental groups.

\begin{corollary} \label{theoremmain} 
Let $X$ be a smooth compact K\"ahler manifold of dimension $n$ and algebraic dimension $a(X)=n-1$. Then $\pi_1(X)$ is projective.
\end{corollary}

{\bf Acknowledgements.} We would like to thank J. Cao, A. Dimca, P. Graf, J. Koll\'ar, and N. Nakayama
for very helpful communications on the various technical problems related to this project.
This work was partially supported by the Agence Nationale de la Recherche grant project Foliage (ANR-16-CE40-0008) and Hodgefun (ANR-16-CE40-0011-02).

\section{Notation and basic definitions}

All complex spaces are supposed to be of finite dimension,  
a complex manifold is a smooth Hausdorff irreducible complex space.
A fibration is a proper surjective morphism with connected fibres between complex spaces.
A fibration $\varphi: X \rightarrow Y$ is locally projective if there exists an open covering 
$U_i \subset Y$ such that $\fibre{\varphi}{U_i} \rightarrow U_i$ is projective, i.e. admits 
a relatively ample line bundle.

We refer to \cite{Gra62, Fujiki78, Dem85} for basic definitions about $(p,q)$-forms
and K\"ahler forms in the singular case.  

\begin{definition} \label{definitionkaehlermorphism}
Let $\holom{\varphi}{X}{Y}$  be a fibration. A relative K\"ahler form is 
a {\em smooth} real closed $(1,1)$-form $\omega$ such that for every $\varphi$-fibre $F$, the restriction
$\omega|_F$ is a K\"ahler form. We say that $\varphi$ is a K\"ahler fibration if such a relative K\"ahler form exists.
\end{definition}

\begin{remark} \label{remarkrelplusbase}
If $\varphi$ is a K\"ahler fibration over a K\"ahler base $Y$, then $X_U := \fibre{\varphi}{U} $ is K\"ahler 
for every relatively compact open set $U \subset Y$. In fact if $\omega_X$ is a relative K\"ahler
form on $X$ and $\omega_Y$ is a K\"ahler form on $Y$, then for all $m \gg 0$ the form $\omega_X + m \varphi^* \omega_Y$
is K\"ahler \cite[Proposition 4.6 (2)]{Bin83}, \cite{Fujiki78}.  
\end{remark}

\begin{definition} \label{definitionapproximation}
Let $X$ be a normal compact K\"ahler space. We say that $X$ admits an algebraic approximation if there exist a flat morphism $\pi: \mathcal X \rightarrow \Delta$ and a sequence $(t_n)_{n \in \N}$ in $\Delta$ converging to $0$ such that $\fibre{\pi}{0}$ is isomorphic to $X$ and $\fibre{\pi}{t_n}$ is a projective variety for all $n$.
\end{definition}

In general a flat deformation does not preserve the fundamental group, this holds however
for deformations that are locally trivial in the sense of \cite[p.627]{FK87}. For the case
of  fibrations we will work with an even more restricted class of deformations:

\begin{definition} \label{definitionlocallytrivial}
Let $\holom{\varphi}{X}{S}$ be a fibration between normal compact complex spaces.
A locally trivial deformation of $(X, \varphi)$ is a pair of fibrations
$$
\pi: \mathcal X \rightarrow \Delta, \qquad \Phi: \mathcal X \rightarrow S \times \Delta
$$ 
such that $\pi = p_S \circ \Phi$, where $\holom{p_S}{S \times \Delta}{\Delta}$ is the projection onto the first factor and the following holds:
\begin{itemize}
\item $X \simeq \fibre{\pi}{0}$ and  $\varphi = \Phi|_{\fibre{\pi}{0}}$;
\item There exists an open cover $(U_i)_{i \in I}$ of $S$ such that (up to replacing $\Delta$ by a smaller polydisc containing $0$) we have
$$
\fibre{\Phi}{U_i \times \Delta} \simeq \fibre{\varphi}{U_i} \times \Delta
$$
for all $i \in I$.
\end{itemize}
\end{definition}

Let us recall some basic definitions on geometric orbifolds introduced in \cite{Cam04}.
They are pairs $(X,\Delta)$ where $X$ is a complex manifold and $\Delta$ a Weil $\Q$-divisor; they appear naturally as bases of fibrations to describe their multiple fibres:
let \holom{\varphi}{X}{Y} be a fibration between compact K\"ahler manifolds and consider $\vert\Delta\vert\subset Y$ the union of the codimension one components of the $\varphi$-singular locus. If $D\subset\vert\Delta\vert$, we can write
$$\varphi^*(D)=\sum_j m_jD_j+R,$$
where $D_j$ is mapped onto $D$ and $\varphi(R)$ has codimension at least 2 in $Y$.

The integer $m(\varphi,D)=\mathrm{gcd}_j(m_j)$ 
is called the classical multiplicity of $\varphi$ above $D$ and we can consider the $\Q$-divisor
\begin{equation} \label{orbifoldbase}
\Delta=\sum_{D\subset \vert\Delta\vert}(1-\frac{1}{m(\varphi,D)})D.
\end{equation}
The pair $(Y,\Delta)$ is called the \emph{orbifold base} of $\varphi$.

\begin{remark} \label{remarkmultiplicitieselliptic}
In Campana's work \cite{Cam04} both the classical and the non-classical multiplicities  $\mathrm{inf}_j(m_j)$ play an important role. Let us note that for elliptic fibrations,
these multiplicities coincide: the problem is local on the base, moreover we can reduce
to the case of a relatively minimal elliptic fibration. Then it is sufficient to observe
that in Kodaira's classification of singular fibres which are not multiple \cite{Kod60}, \cite[Chapter V, Table 3]{BHPV04}, there is always
at least one irreducible component of multiplicity one.  
\end{remark}

Let us recall what smoothness means for a geometric orbifold.

\begin{definition}
A geometric orbifold $(X/\Delta)$ is said to be smooth if the underlying variety $X$ is a smooth manifold and if the $\Q$-divisor $\Delta$ has only normal crossings. If in a coordinate patch, the support of $\Delta$ can be defined by an equation
$$\prod_{j=1}^r z_j=0,$$
we will say that these coordinates are adapted to $\Delta$.
\end{definition}

In the category of smooth orbifolds, there is a good notion of fundamental group. It is defined in the following way : if $\Delta=\sum_{j\in J}(1-\frac{1}{m_j})\Delta_j$, choose a small loop $\gamma_j$ around each component $\Delta_j$ of the support of $\Delta$. 
Consider now the fundamental group of $X^\star=X\backslash \textrm{Supp}(\Delta)$ and its normal subgroup generated by the loops $\gamma_j^{m_j}$:
$$\langle\!\langle \gamma_j^{m_j},\,j\in J \rangle\!\rangle \leq \pi_1(X^\star).$$

\begin{definition} \label{defipi1orbifold}
The fundamental group of $(X/\Delta)$ is defined to be:
$$\pi_1(X/\Delta):=\pi_1(X^\star)/\langle\!\langle \gamma_j^{m_j},\,j\in J \rangle\!\rangle.$$
\end{definition}

\begin{remark} \label{remarkp1orbifold}
By definition the loops $\gamma_j$ define torsion elements in  $\pi_1(X/\Delta)$. Thus we see that if
$\pi_1(X/\Delta)$ is torsion-free, the natural surjection $ \pi_1(X/\Delta) \twoheadrightarrow \pi_1(X)$
is an isomorphism.
\end{remark}

\section{Elliptic fibrations}


The structure of elliptic fibrations and their deformation theory has been described
in detail in the landmark paper of Kodaira \cite{Kod60} for surfaces and its generalisation 
to higher dimension by Nakayama \cite{NakLocal, Nak02c}. For the convenience of
the reader we review this theory and explain some additional properties that
will be important in the proof of Theorem \ref{theoremapproximation}.

\subsection{Smooth fibrations and their deformations} 
\label{subsectiondeformations}

Let $S^\star$ be a complex manifold, and let $\holom{f^\star}{X^\star}{S^\star}$ be a smooth elliptic fibration. We associate a variation of Hodge structures $H$ (\textsc{vhs} for short in the sequel) of weight 1 over $S^\star$, the underlying local system being given by the first cohomology group of the fibers $H^1(X_s,\Z)\simeq \Z^2$.
A rank 2 and weight 1 \textsc{vhs} over $S^\star$ is equivalent to the following data: a holomorphic function to the upper half plane
$$
\tau_H:\widetilde{S^\star}\to \mathbb{H}
$$
defined on the universal cover $\widetilde{S^\star} \rightarrow S^\star$
which is equivariant under a representation
$$
\rho_H:\pi_1(S^\star)\To \mathrm{SL}_2(\Z).
$$
Now if $\gamma\in\pi_1(S^\star)$, let us write
$$\rho_H(\gamma)=\left(\begin{array}{cc}a_\gamma & b_\gamma \\ c_\gamma & d_\gamma
\end{array}\right)$$
the image of $\gamma$ under $\rho_H$. It is then straightforward to check that the following formula
$$((m,n),\gamma)\cdot (x,z)=\left(\gamma(x),\frac{z+m\tau_H(x)+n}{c_\gamma\tau_H(x)+d_\gamma}\right)$$
defines an action of the semi-direct product $\Z^2\rtimes \pi_1(S^\star)$ on $\tilde{S^\star}\times \C$
which is fixed point free and properly discontinuous. We can then form the quotient to get
a smooth elliptic fibration
$$
p:\mathbf{J}(H)\To S^\star
$$
which is endowed with a canonical section $\sigma:S^\star\to \mathbf{J}(H)$.
Following the terminology of \cite{Kod, Nak02c} we call $p$ the 
\emph{basic elliptic fibration} (associated to $H$). Note that 
$f^\star$ and $p$ are locally isomorphic over $S^\star$, but their global structure 
can be quite different.

Since $p$ has a global section, its sheaf of holomorphic sections $\mathcal{J}(H)$ is a 
well-defined sheaf of abelian groups and we have an exact sequence of sheaves
\begin{equation}\label{eq:exp sequence smooth}
0\To H\To \mathcal{L}_H\To \mathcal{J}(H)\To 0
\end{equation}
where 
$$
\mathcal{L}_H := R^1p_*\mathcal{O}_{\mathbf{J}(H)} \simeq R^1f^\star_* \mathcal{O}_{X^\star}.
$$
Let us note that $\mathcal{L}_H$ can also be interpreted as the zeroth graded piece $\cH/F^1\cH$ of the Hodge filtration on $\cH \colonec H \otimes \cO_{S^\star}$ and thus depends only on the \textsc{vhs} $H$. Since $f^\star$ is smooth, it has local sections over every point of $S^\star$ and the difference of two such sections on the intersection of their sets of definition can be seen as a section of $\mathcal{J}(H)$. In this way we have just associated a cohomology class
$$
\eta(f^\star)\in H^1(S^\star,\mathcal{J}(H))
$$
 to $f^\star$, which is independent of the choices of local sections.
The class $\eta(f^\star)$ can be also constructed in a more conceptual way. Pushing forward the exponential sequence on $X^\star$ by $f^\star$ yields a long exact sequence on $S^\star$:
\begin{equation}\label{eq:push forward exp sequence}
0\To R^1f^\star_*\Z_X\To R^1f^\star_*\mathcal{O}_X\To R^1f^\star_*\mathcal{O}^*_X
\To R^2f^\star_*\Z_X\simeq \Z_{S^\star} \To 0
\end{equation}
Recall that $\mathcal{J}(H) =R^1f^\star_*\mathcal{O}_X/R^1f^\star_*\Z_X$, the class $\eta(f^\star)$ is the image of $1\in H^0(S^\star,\Z_{S^\star})$ under the connecting morphism 
$$
\delta :H^0(S^\star,\Z_{S^\star})\To H^1(S^\star,\mathcal{J}(H)).
$$
The map $\eta$ is in fact a bijection.
\begin{theorem} \label{theoremsmoothtorsor} \cite[Theorem 10.1, Theorem 11.5]{Kod}
\cite[Proposition 1.3.1, Proposition 1.3.3]{NakLocal}
Let $S^\star$ be a complex manifold, and let $H$ be a \textsc{vhs} of rank 2 and weight 1 over $S^\star$. 
\begin{enumerate}[(a)]
\item The map $f^\star \mapsto \eta(f^\star)$ defines a one-to-one correspondence between the isomorphism classes of 
smooth elliptic fibrations over $S^\star$ inducing $H$ and the elements of the cohomology group $H^1(S^\star,\mathcal{J}(H))$. 
\item A smooth elliptic fibration  $f^\star:X^\star\to S^\star$ is a projective morphism 
if and only if $\eta(f^\star)$ is a torsion class. 
\end{enumerate}
\end{theorem}

The short exact sequence \eqref{eq:exp sequence smooth}
induces an exact sequence
\begin{equation}\label{eq:def space H^1}
H^1(S^\star,\mathcal{L}_H)\stackrel{\mathrm{exp}}{\To}H^1(S^\star,\mathcal{J}(H))\stackrel{\mathtt{c}}{\To} H^2(S^\star,H).
\end{equation}
The vector space $V \colonec H^1(S^\star,\mathcal{L}_H)$ appears as a deformation space
of smooth elliptic fibrations over $S^\star$. More precisely, given an elliptic fibration $f^\star:X^\star\to S^\star$ inducing $H$, there exists a family of elliptic fibrations $\Pi : \cX \to S^\star \times V$ over $S^\star$ parameterized by $V$ such that the fiber over $t\in H^1(S^\star,\mathcal{L}_H)$ is an elliptic fibration whose associated element in $H^1(S^\star,\mathcal{J}(H))$ is $\exp(t) + \eta(f^\star)$. Viewing $\Pi$ as a smooth elliptic fibration, the cohomology class $\eta(\Pi) \in H^1(S^\star \times V,\mathcal{J}(\pr^*H))$ associated to $\Pi$ is equal to  $\exp(\xi) + \pr^*\eta(f^\star)$, where $\pr : S^\star \times V \to S^\star$ denotes the projection onto the first factor and 
$$\xi \in H^1(S^\star ,\mathcal{L}_{H}) \otimes H^0(V,\mathcal{O}_{V}) \subset H^1(S^\star \times V,\mathcal{L}_{\pr^*H})$$ 
the element which corresponds to the identity map $V \to H^1(S^\star ,\mathcal{L}_{H})$.

\begin{theorem}\label{theoremsmoothdeformation} \cite[Proposition 2.4 and 2.5]{Cla16}
\begin{enumerate}[(a)]
\item Let $\holom{f_1^\star}{X_1^\star}{S^\star}$ and $\holom{f_2^\star}{X_2^\star}{S^\star}$
be two smooth elliptic fibrations over $S^\star$ inducing $H$.
Then $f_1^\star$ can be deformed into $f_2^\star$ in the family described above if and only if 
$\mathtt{c}(\eta(f_1^\star))=\mathtt{c}(\eta(f_2^\star))$
\item Let $f^\star:X^\star\to S^\star$ be a smooth elliptic fibration such that $X^\star$ is K\"ahler.
Then $\mathtt{c}(\eta(f^\star))$ is torsion and $f^\star$ can be deformed to a projective fibration.
\end{enumerate}
\end{theorem}

\subsection{Local structures and Weierstra{\ss} models}\label{subs:local_structure}

\begin{definition} \label{definitionelliptic}
An elliptic fibration $\holom{f}{X}{S}$ is a fibration whose general fibre $X_s$ is isomorphic to an elliptic curve. The elliptic fibration $f$
is relatively minimal if $K_X$ is $\Q$-Cartier and there exists a line bundle $L$ on $S$ such that
$$
m K_X \simeq f^* L
$$
for some $m \in \N^*$. 

We say that $f$ has a meromorphic section in a point $s \in S$ if there exists an analytic neighbourhood $s \in U \subset S$ and a meromorphic map $\merom{s}{U}{X}$ 
such that $f \circ s$ is the inclusion $U \hto S$.
\end{definition}

Apart from smooth elliptic fibrations discussed in the last section, the second simplest examples of elliptic fibrations are Weierstra\ss~ fibrations. These fibrations turn out to be crucial in the study of elliptic fibrations.

\begin{definition} \label{definitionweierstrass}
Let $S$ be a complex manifold. 
\begin{enumerate}[(a)]
\item A Weierstra\ss~ fibration over $S$  
consists of a line bundle $\mathcal{L}$ on $S$ and two sections $\alpha\in H^0(S,\mathcal{L}^{(-4)})$ and $\beta\in H^0(S,\mathcal{L}^{(-6)})$ such that $4\alpha^3+27\beta^2$ is a non zero section of $H^0(S,\mathcal{L}^{(-12)})$. With these data, we can associate a projective family of elliptic curves:
$$
\mathbb{W}:=\mathbb{W}(\mathcal{L},\alpha,\beta)=\set{Y^2Z=X^3+\alpha XZ^2+\beta Z^3}
\subset \PP
$$
where 
$$
\PP:=\PP(\mathcal{O}_S\oplus \mathcal{L}^2\oplus \mathcal{L}^3)
$$
and $X$, $Y$ and $Z$ are canonical sections of $\mathcal{O}_\PP(1)\otimes \mathcal{L}^{(-2)}$, $\mathcal{O}_\PP(1)\otimes \mathcal{L}^{(-3)}$ and $\mathcal{O}_\PP(1)$ respectively.
The restriction of the natural projection $\PP\to S$ to $\mathbb{W}$ gives rise to a flat morphism
$p_\mathbb{W}:\mathbb{W}\to S$
whose fibres are irreducible cubic plane curves. This elliptic fibration is endowed with a distinguished section $\set{X=Z=0}$.
\item A Weierstra{\ss} fibration $\mathbb{W}(\mathcal{L},\alpha,\beta)$ is said to be minimal if there is no prime divisor $\gD \subset S$ such that $\mathrm{div}(\alpha) \ge 4 \gD$ and $\mathrm{div}(\beta) \ge 6 \gD$.
\item A (minimal) locally Weierstra\ss~fibration is an elliptic fibration $\holom{f}{X}{S}$ such that there exists an open covering $(U_i)_{i \in I}$
of $S$ such that the restriction of $f$ to $X_i:=\fibre{f}{U_i}$ is a (minimal) Weierstra\ss~fibration.
\end{enumerate}
\end{definition}


\begin{remark} \label{remarkpropertiesweierstrass}
Since the total space of the Weierstra\ss~ fibration $p_\mathbb{W}:\mathbb{W}\to S$ is by definition a hypersurface in a manifold, the complex space $\mathbb{W}$ is Gorenstein, so the canonical sheaf is locally free. In the case where $\mathbb{W}(\mathcal{L},\alpha,\beta)$ is minimal and the discriminant divisor $\mathrm{div}(4\alpha^3+27\beta^2)$ has normal crossings, we know by \cite[Corollary 2.4]{NakW} that $\mathbb{W}$ has rational, hence canonical, singularities (and it holds of course for the total space of a locally Weierstra\ss~ fibration).
Since $p_\mathbb{W}$ is flat and the restriction of $K_{\mathbb{W}}$ to every fibre is trivial, we have $K_{\mathbb{W}} \simeq p^*_\mathbb{W} L$ for some line bundle $L$ on $S$. Thus a locally Weierstra\ss~ fibration is relatively minimal.
\end{remark}

It is well-known that smooth elliptic fibrations with a section are always Weierstra\ss:

\begin{theorem}\label{th:rep weierstrass smooth} \cite{Kod} \cite[Proposition 1.2.4]{NakLocal} Let $S$ be a complex manifold, 
and let $\holom{f}{X}{S}$ be a smooth elliptic fibration admitting a section $\holom{s}{S}{X}$.
Then there exists a canonically defined isomorphism $X \rightarrow \mathbb{W}$ over $S$ to some Weierstra\ss~ fibration
$\mathbb{W} \rightarrow S$ sending $s$ onto the distinguished section.
\end{theorem}

For non-smooth fibrations we can only hope to work with bimeromorphic models:

\begin{definition} \label{definitiongenweierstrassmodel}
Let $\holom{f}{X}{S}$ be an elliptic fibration and let $\holom{f^\star}{X^\star}{S^\star}$ denote the restriction to a non-empty Zariski open subset $S^\star \subset S$ such that $f$ is smooth. 
A  (locally) Weierstra\ss~model of $f$ is a (locally) Weierstra\ss~fibration $p : \mathbb{W} \rightarrow S$ such that $X^\star$ is isomorphic to $\W^\star \colonec p^{-1}(S^\star)$ over $S$.
\end{definition}

The existence of meromorphic sections is an obvious necessary condition for the existence of a Weierstra\ss~model. When the base $S$ is smooth, it is also sufficient:

\begin{theorem} \label{theoremlocalproperties}  \cite[Theorem 2.5]{NakW}
Let $S$ be a complex manifold, and let $\holom{f}{X}{S}$ be an elliptic fibration.
If $f$ admits a meromorphic section, then 
$f$ has a unique minimal Weierstra\ss~model.
\end{theorem}

For a similar statement concerning locally Weierstra\ss~model, we refer to Remark~\ref{rem-locW}.

The following vanishing result will be useful.

\begin{theorem}\label{thm-van} \cite[Theorem 3.2.3]{NakLocal}
Let $\holom{f}{X}{S}$ be an elliptic fibration such that both $X$ and $S$ are smooth and that $f$ is smooth over the complement of a normal crossing divisor in $S$. Then we have
$$
R^j f_* \sO_X = 0 \qquad \forall \ j \geq 2.
$$
\end{theorem}

The following example explains the importance of the normal crossing condition for the theory of elliptic fibrations: 

\begin{example} Let $S$ be a smooth non-algebraic compact K\"ahler surface that admits an elliptic fibration $\holom{g}{S}{\PP^1}$. 
Let $\mathbb F_1 \rightarrow \PP^1$ be the first Hirzebruch surface, and set $X:= \mathbb F_1 \times_{\PP^1} S$.
Then $X$ is a smooth compact K\"ahler threefold, and we denote by
$f: X \rightarrow \PP^2$ the composition of the elliptic fibration $X \rightarrow \mathbb F_1$ 
with the blowdown $\mathbb F_1 \rightarrow \PP^2$.

Then $f$ is not locally projective since it has a two-dimensional fibre isomorphic
to the non-projective surface $S$. Note however that $g$ has at least $3$ singular fibres
(\cite[Proposition 1]{Bea81}, cf. Proposition \ref{propositionbeauvillekaehler} for a detailed proof in the 
K\"ahler case). Thus the discriminant locus of $f$ consists of at least $3$ lines
meeting in one point. In particular it is not a normal crossing divisor and it is quite easy to see that $R^2f_*\sO_X\neq0$ in this case.
\end{example}

A general elliptic fibration does not admit local meromorphic sections at every point, a fact that is the starting point of Nakayama's global theory of elliptic fibration using the $\partial$-\'etale cohomology.
For our needs we can use the strategy of Kodaira \cite{Kod} to reduce to this case
via base change:

\begin{proposition}\label{prop:local meromorphic section}
Let $\holom{f}{X}{S}$ be an elliptic fibration such that both $X$ and $S$ are smooth and that $f$ is smooth over the complement of a normal crossing divisor in $S$. Suppose that $f$ is locally projective (e.g. when $X$ is a K\"ahler manifold~\cite[Theorem 3.3.3]{NakLocal}) and $S$ is projective, then there exists a finite Galois cover $\tilde S \rightarrow S$ by some projective manifold $\tilde S$ such that  
$$
X \times_S \tilde S \rightarrow \tilde S
$$
has local meromorphic sections over every point of $\tilde{S}$. The elliptic fibration $X \times_S \tilde S \rightarrow \tilde S$ is smooth over the complement of a normal crossing divisor in $\tilde S$.
\end{proposition}

This statement is a variant of \cite[Corollary 4.3.3]{NakLocal}: in our case $S$ is projective, but
we lose the control over the branch locus.

\begin{proof}
For every irreducible component $D_i$ of $D$ we denote by $m_i \in \N$ the multiplicity of the generic
fibre over $D_i$. By \cite[Proposition 4.1.12]{Laz04a} we can choose a covering $\tilde S \rightarrow S$ ramifying with multiplicity
exactly $m_i$ over $D_i$ and the ramification divisor has normal crossings. 
By construction the elliptic fibration $X \times_S \tilde S \rightarrow \tilde S$ has no multiple fibre in codimension one.
Up to taking another finite cover and the Galois closure we can suppose that  $\tilde S \rightarrow S$
is Galois and the local monodromies are unipotent.
Since the elliptic fibration is locally projective, we can 
now apply \cite[Theorem 4.3.1 and 4.3.2]{NakLocal} to conclude that it has local meromorphic sections over every point of $\tilde S$.
\end{proof}

\subsection{Elliptic fibrations with local meromorphic sections}

In this subsection we always work under the following 

\begin{assumption} \label{assume}
Let $S$ be a complex manifold, and let $\holom{f}{X}{S}$ be an elliptic fibration having
local meromorphic sections over every point of $S$.
We denote by $j: S^\star \subset S$ a Zariski open subset such that $X^\star := \fibre{f}{S^\star}
\rightarrow S^\star$ is smooth and assume that the complement $S \setminus S^\star$ is a normal crossing divisor.  
\end{assumption}

Denote by $H$ the \textsc{vhs} on $S^\star$ induced by the smooth elliptic
fibration $X^\star \rightarrow S^\star$.
We set
$$
\mathcal{L} := \mathcal{L}_{H/S} := R^1 f_* \sO_X.
$$

Let $p : \mathbf J(H) \rightarrow S^\star$ be the basic elliptic fibration associated 
with $H$. By Theorem~\ref{theoremlocalproperties}, we can extend $p$ to a  Weierstra\ss~
model
$$
p_{\mathbb W}: \mathbb W \rightarrow S.
$$
We also denote by $p: \mathbf B(H) \rightarrow S$ the composition of $p_{\mathbb W}$
with a desingularisation $\mathbf B(H) \rightarrow \mathbb W$. Since $p_{\mathbb W}$ has
a canonical section, the elliptic fibration $p$ has a global meromorphic section.
We call $p$ a basic elliptic fibration associated to $H$ \cite[p.549]{Nak02c}.

By~\cite[Lemma 3.2.3]{NakLocal},  $\mathcal{L}_{H/S}$ is isomorphic to the zeroth graded piece of the Hodge filtration of the lower canonical extension of $\cH = H \otimes \cO_{S^\star}$ to $S$. This induces a natural map $j_* H \rightarrow \mathcal L_{H/S}$, which is injective by~\cite[Lemma 3.1.3]{NakLocal}. Let $\mathcal{J}(H)^\mathbb{W}$ denote the quotient $\mathcal L_{H/S} / j_* H$. The exact sequence 
\begin{equation}\label{exseq-JW}
0 \rightarrow j_* H \rightarrow \mathcal L_{H/S} \rightarrow \mathcal{J}(H)^\mathbb{W} 
\rightarrow 0
\end{equation}
extends the exact sequence \eqref{eq:exp sequence smooth} defined
on $S^\star \subset S$. 

Let $\W^\# \subset \W$ denote the Zariski open of $\W$ consisting of points $x \in \W$ where $p : \W \to S$ is smooth. The variety $\W^\#$ is a complex analytic group variety over $S$, where over a point $t \in S$ which parameterises a nodal (resp. cuspidal) rational curve in $p: \W \to S$, the fiber is the multiplicative group $\bC^\times$ (resp. additive group $\bC$). 
If we denote by $\holom{\nu}{\PP^1 = \bC \cup \{\infty\}}{\W_t}$ the normalisation so that $\infty$ maps to the unique singular
point, the multiplication by $m \in \Z \setminus \{ 0 \}$ on the smooth part of $\W_t$ is induced by
$x \mapsto x^m$ (if $\W_t$ is nodal) or $x \mapsto mx$ (if $\W_t$ is cuspidal).
In both cases a neighbourhood of $\infty$ in $\PP^1$ is mapped to a neighbourhood of $\infty$.

The action of $\W^\#$ on itself extends to an action on $\W$~\cite[Lemma 5.1.1 (7)]{Nak02c}, but the group variety structure does not extend to $\W$. However, the multiplication by $m$ extends to the whole $\W$.

\begin{lemma}\label{lem-multm}
Given a non-zero integer $m \in \bZ \setminus \{ 0 \}$, the multiplication-by-$m$ map $\W^\# \to \W^\#$ extends to a finite holomorphic map $\bm : \W \to \W$. 
\end{lemma}

\begin{proof}
Fix a point $x \in \W \bss \W^\#$. Locally around the fiber $p^{-1}(p(x))$, there exists a polydisc $U \subset \W$ containing $x$ such that $\W$ is a hypersurface of $\PP^2 \times U$ and that each fiber of $\W_{U} \colonec p^{-1}(U) \to U$ is a cubic curve in $\PP^2$. 
The description of the multiplication given above shows 
that  the closure in $\W_{U}$ of the image of a neighborhood of $(\W \bss \W^\#) \cap \W_{U}$ in $\W_{U}$ is still a neighborhood of $(\W \bss \W^\#) \cap \W_{U}$. As $\W$ is normal, we can thus apply the Riemann extension theorem~\cite[p. 144]{GR84} to conclude that the multiplication-by-$m$ map $\W^\# \to \W^\#$  extends to a holomorphic map $m : \W \to \W$. As $m_{|\W^\#}$ preserves fibers of $\W^\# \to S$ and is finite, the extension $m : \W \to \W$ is finite. 
\end{proof}

\begin{remark}
In~\cite[p. 550]{Nak02c}, the sheaf $\mathcal{J}(H)^\mathbb{W}$ is first \emph{defined} to be the germs of holomorphic sections of $p : \W \to S$, then one proves that $\mathcal{J}(H)^\mathbb{W}$ sits inside the exact sequence~\eqref{exseq-JW}. However with this definition of $\mathcal{J}(H)^\mathbb{W}$, the exactness of~\eqref{exseq-JW} fails as it follows from the false claim that local sections of $p$ are contained in $\W^\#$. Indeed, the Weierstra{\ss} fibration  parameterized by $\ga \in \bC$ defined by $Y^2Z = X^3 + \ga X$ has a section $\ga \mapsto (X(\ga) = 0, Y(\ga) = 0)$ which passes through the cusp of the singular central fiber. 
 
In order to keep the sequence~\eqref{exseq-JW} exact, the correct definition of $\mathcal{J}(H)^\mathbb{W}$ should be the sheaf of germs of holomorphic sections of $\W \to S$ whose image is \emph{contained in $\W^\#$}. In this way, as already mentioned in~\cite[p. 550]{Nak02c} since $\W^\#$ acts on $\W$ by translations~\cite[Lemma 5.1.1 (7)]{Nak02c}, a local section of  $\mathcal{J}(H)^\mathbb{W}$ gives rise to a local automorphism of $\W$.
\end{remark}

We can associate an elliptic fibration to a cohomology class $\eta\in H^1(S,\mathcal{J}(H)^\mathbb{W})$~\cite[p.550]{Nak02c}: 

\begin{construction} \label{constructelliptic}
Fix an open cover $(U_j)_{j \in \N}$ of $S$ such that the class $\eta$ is represented by a cocycle $(\eta_{ij})_{i<j}$ where $\eta_{ij}\in H^0(U_i \cap U_j,\mathcal{J}(H)^\mathbb{W})$.
By the remark above, with the choice of a zero-section $U_i \to \W^\#|_{U_i}$ for each $i$, we can identify the $\eta_{ij}$ to automorphisms of $\mathbb \W_{ij} \colonec \mathbb \W|_{U_i \cap U_j}$ over $S$. The cocycle condition assures that the condition
of the gluing lemma \cite[Chapter II, Exercise 2.12]{Har77} is satisfied in our situation, so we can glue the elliptic fibrations $\mathbb \W_{i} \colonec \mathbb{W}_{|U_i} \rightarrow U_i$ to an elliptic fibration $p^\eta: \mathbb{W}^\eta \rightarrow S$. Since the gluing morphisms are translations so act as the identity on the \textsc{vhs},
the \textsc{vhs} induced by $p^\eta$ on $S^\star$ is $H$. This construction is independent of the choices of $(U_i)$ and the zero-sections $U_i \to \W^\#|_{U_i}$
\end{construction}

According to the above construction, given $p^\eta: \mathbb{W}^\eta \rightarrow S$ and an open cover $(U_i)$ of $S$ as above, the multiplication-by-$m$'s on $ \W|_{U_i} \to U_i$ defined in Lemma~\ref{lem-multm} glue together to a global morphism $\bm : \mathbb{W}^\eta \to \mathbb{W}^{m\eta}$ over $S$, which up to isomorphisms is independent of the choices of $(U_i)$ and the zero-sections $U_i \to \W^\#|_{U_i}$.

Now given a locally Weierstrass fibration $f : X  \to S$ constructed by~\ref{constructelliptic}, we shall explain how to find $\eta$ to which $f$ associates.
Consider the long exact sequence
 \begin{equation}\label{sel}
\begin{tikzcd}[cramped, row sep = 4, column sep = 20]
\cdots \arrow[r] & R^1f_*\bZ_X \arrow[r] & R^1f_*\cO_X \arrow[r] & R^1f_*\cO_X^* \arrow[r] & R^2f_*\bZ_X \arrow[r]  & \cdots.  \\ 
\end{tikzcd}
\end{equation}
Since $X \to S$ is obtained by gluing the pieces $\mathbb{W}_i \rightarrow U_i$ by translation maps $\tau_{ij} : \mathbb{W}_{ij} \to \mathbb{W}_{ij}$, which act trivially on $H^1(\mathbb{W}_{ij}, \bZ)$, we have $R^1f_*\bZ_X = R^1p_*\bZ_\W \simeq j_*H$. The translations $\tau_{ij}$ act also trivially on $H^1(\W_s, \cO_{\W_s})$ where $\W_s \colonec p^{-1}(s)$ for any $s \in U_i \cap U_j$. As $p : \W \to S$ is flat and $ H^1(\W_s, \cO_{\W_s}) \simeq\bC$, by Grauert's base change theorem we deduce that  $R^1f_*\cO_X = R^1p_*\cO_\W$, which is isomorphic to $\cL_{H/S}$. Therefore $R^1f_*\bZ_X \to R^1f_*\cO_X$ is identical to the morphism $j_*H \to \cL_{H/S}$ in~\eqref{exseq-JW}.
Finally since the fibers of $f$ are of dimension 1, we have $R^2f_*\cO_X = 0$. Thus~\eqref{sel} becomes
 \begin{equation}\label{ex-long}
\begin{tikzcd}[cramped, row sep = 4, column sep = 20]
j_*H \arrow[r] & \cL_{H/S} \arrow[r] &R^1f_*\cO_X^* \arrow[r] & R^2f_*\bZ_X \arrow[r]  & 0.  \\ 
\end{tikzcd}
\end{equation}
Recall that $j_*H \to \cL_{H/S}$ is injective and $\cJ(H)^\W$ sits inside the short exact sequence
 \begin{equation}\label{ex-1}
\begin{tikzcd}[cramped, row sep = 4, column sep = 20]
0 \arrow[r] & j_*H \arrow[r] & \cL_{H/S} \arrow[r] & \cJ(H)^\W  \arrow[r] & 0.  \\ 
\end{tikzcd}
\end{equation}
 As a fiber $F$ of $f$ is either an elliptic curve, a nodal rational curve, or a rational curve with a cusp, we have $H^2(F,\bZ) = \bZ$. Since $p$ is proper, by~\cite[Theorem 6.2]{Ive} $R^2f_*\bZ_X \simeq \bZ_S$. Hence we have a second short exact sequence
 \begin{equation}\label{ex-2}
\begin{tikzcd}[cramped, row sep = 4, column sep = 20]
0 \arrow[r] & \cJ(H)^\W \arrow[r]   & R^1f_*\cO_X^* \arrow[r]   & \bZ_S  \arrow[r] & 0.  \\ 
\end{tikzcd}
\end{equation}
If $\eta \in H^1(S, \cJ(H)^\W)$ denotes the element which defines~\eqref{ex-2}, then $f$ will be the elliptic fibration associated to $\eta$.

For classes in $H^1(S,\mathcal{J}(H)^\mathbb{W})$ the deformation theory is analogous to the smooth case:

\begin{proposition}\label{prop:deformation loc trivial}
Assuming~\ref{assume}. Given $\eta \in H^1(S,\mathcal{J}(H)^\mathbb{W})$, there exists a locally trivial (cf. Definition \ref{definitionlocallytrivial}) family of elliptic fibrations $\pi:\sX\To S\times V$ over $S$ parameterised by $V \colonec H^1(S,\mathcal{L})$ satisfies the following property: an elliptic fibration $X \to S$ is a member of $\pi$ if and only if $X$ is isomorphic to $W^\theta \to S$ over $S$ for some $\theta \in H^1(S,\mathcal{J}(H)^\mathbb{W})$ such that $\mathtt{c}(\eta)=\mathtt{c}(\theta)$.
\end{proposition}

\begin{proof}
As in the smooth case, let $$\xi \in H^1(S ,\mathcal{L}) \otimes H^0(V,\mathcal{O}_{V}) \subset H^1(S \times V,\mathcal{L}_{\pr^*H/{S \times V}})$$ 
be the element which corresponding to the identity map $V \to H^1(S ,\mathcal{L}_{H/S})$ where $\pr : S \times V \to S$ denotes the projection onto the first factor. Let $\pi:\sX \to S\times V$ be the elliptic fibration obtained by Construction~\ref{constructelliptic} from $\exp(\xi) + \pr^*\eta \in  H^1(S \times V,\mathcal{J}(\pr^*H)^\W)$. Then considering $\pi$ as a family of elliptic fibrations over $S$ parameterised by $V$, the fiber over $t\in H^1(S,\mathcal{L}_{H/S})$ is the elliptic fibration constructed by~\ref{constructelliptic} from $\exp(t) + \eta \in H^1(S,\mathcal{J}(H))$. Thus $\pi$ satisfies the desiring property. As $V$ is contractible, in order to construct $\pi:\sX \to S\times V$, it is possible to take the open cover of $S \times V$ in Construction~\ref{constructelliptic} to be $\{U_i \times V \}$ for some open cover $\{U_i\}$ of $S$. Thus $\pi:\sX \to S\times V$ is locally trivial.
\end{proof}

The cohomology group $H^1(S,\mathcal{J}(H)^\mathbb{W})$ is a parameter set of elliptic fibrations over $S$ with \textsc{vhs} $H$, but for classification purposes it is too small. We denote by $\mathcal{J}(H)_{mer}$ the sheaf of {\em meromorphic} sections of a basic
elliptic fibration $p: \mathbf{B}(H)\to S$. Since two basic elliptic fibrations
are bimeromorphically equivalent the sheaf $\mathcal{J}(H)_{mer}$ does not depend on the choice of the model.
Moreover, since $p$ has a global meromorphic section, we see that
$\mathcal{J}(H)_{mer}$ has a group structure \cite[p. 243-244]{NakLocal}. There is a trivial inclusion of sheaves of abelian groups
\begin{equation}
\mathcal{J}(H)^\mathbb{W} \subset \mathcal{J}(H)_{mer} 
\end{equation}
which is an isomorphism on $S^\star$: since $\mathbb W$ is smooth over $S^\star$ we have
$\mathbf J(H) \simeq \mathbf B(H)|_{S^\star} \simeq \mathbb W|_{S^\star}$, moreover any
meromorphic section is holomorphic over $S^\star$ \cite[Lemma 1.3.5]{NakLocal}.
In particular the quotient sheaf
$$
\mathcal{Q}_H:=\mathcal{J}(H)_{mer}/\mathcal{J}(H)^\mathbb{W}
$$
is supported on $D=S\setminus S^\star$. By \cite[Theorem 5.4.9]{Nak02c} we have
a commutative diagram
\begin{equation}\label{eq:diagram}
\xymatrix{&&&0\ar[d]&&& \\
0\ar[r] & j_*H\ar[r] &\mathcal{L}_{H/S} \ar[r] & \mathcal{J}(H)^\mathbb{W}\ar[d]\ar[r]&0&&\\
&&0\ar[r] & \mathcal{J}(H)_{mer} \ar[r]^{\Psi_f}\ar[d] & R^1f_*\mathcal{O}^*_X/\mathcal{V}_X\ar[r]\ar[d] & \Z_S\ar[r]\ar[d] & 0\\
&&0\ar[r] & \mathcal{Q}_H\ar[r]\ar[d] & R^2f_*\Z_X/\mathcal{V}_X\ar[r] & \Z_S\ar[r] & 0\\
&&&0&&&
}
\end{equation}
where
$$\mathcal{V}_X:=\mathrm{Ker} \left( R^1f_*\mathcal{O}^*_X\To j_*((R^1f_*\mathcal{O}^*_X)_{\mid S^\star})\right)$$
and $\Psi_f$ is constructed from the local meromorphic sections of $f$.

\begin{definition}\label{defi:eta class general}
We define 
$$
\eta(f)\in H^1(S,\mathcal{J}(H)_{mer})
$$
to be the image of $1\in H^0(S,\Z_S)$ under the connecting morphism of the long exact sequence associated to the second line of Diagram (\ref{eq:diagram}).
\end{definition}
By \cite[Proposition 5.5.1]{Nak02c} we have an injection
$$
\mathcal E_0(S, D, H) \hookrightarrow H^1(S,\mathcal{J}(H)_{mer}),
$$
where $\mathcal E_0(S, D, H)$ is the set of bimeromorphic equivalence classes of elliptic fibrations $f : X \rightarrow S$
having meromorphic sections over every point of $S$
and such that $f^{-1}(S^\star) \to S^\star$ is bimeromorphic to a smooth elliptic fibration over $S^\star$ inducing the \textsc{vhs} $H$. By Construction \ref{constructelliptic} we have
$$
H^1(S,\mathcal{J}(H)^\mathbb{W}) \rightarrow \mathcal E_0(S, D, H) \hookrightarrow H^1(S,\mathcal{J}(H)_{mer})
$$
but contrary to the smooth case it is not clear if the images coincide. 
If $S$ is a curve, the skyscraper sheaf $\mathcal{Q}_H$ has no higher cohomology
so the map
$$
H^1(S,\mathcal{J}(H)^\mathbb{W})\To H^1(S,\mathcal{J}(H)_{mer})
$$
is surjective. 

If $\eta(f) \in H^1(S,\mathcal{J}(H)_{mer})$ is the image of some $\eta \in H^1(S,\mathcal{J}(H)^\mathbb{W})$, then there is a morphism of short exact sequences 
\begin{equation}\label{diag-p}
\begin{tikzcd}[cramped, row sep = 20, column sep = 20]
0 \arrow[r] & \cJ(H)^\W \arrow[r]  \ar[d] & R^1p_*\cO_{\W}^* \arrow[r]  \ar[d] & \bZ_S  \arrow[r]  \arrow[d, "\wr"] & 0  \\ 
0 \arrow[r] &  \cJ(H)_{mer} \arrow[r] & R^1f_*\cO_X^*/\cV_X \arrow[r] &  \bZ_S  \arrow[r] & 0  \\ 
\end{tikzcd}
\end{equation}
where the first row is the short exact sequence~\eqref{ex-2} defined by $p^\eta : \W^\eta \to S$.

\subsubsection{The K\"ahler case}

From now on we will focus on the case where the total space of the elliptic fibration $f$ is compact K\"ahler. In that case, the element  $\eta(f)\in H^1(S,\mathcal{J}(H)_{mer})$ represented by $f$ lies in the image of $H^1(S,\mathcal{J}(H)^\mathbb{W})$, up to replacing $\eta(f)$ by a larger multiple
(cf. \cite[Proposition 7.4.2]{Nak02c} for a more general statement).

\begin{lemma}\label{lem:Kahler implies virtually weierstrass}
In the situation of Assumption \ref{assume}, suppose also that $X$ is bimeromorphic to a compact K\"ahler manifold.
Then the image of the class $\eta(f)\in H^1(S,\mathcal{J}(H)_{mer})$ is torsion in $H^1(S,\mathcal{Q}_H)$. In particular there exists an integer $m \ge1$ such that
$$
m\cdot \eta(f)\in H^1(S,\mathcal{J}(H)^\mathbb{W}).
$$
\end{lemma}

\begin{proof}
Since the class $\eta(f)$ depends only on the bimeromorphic equivalence
class of $X \rightarrow S$ we can suppose that $X$ is a compact K\"ahler manifold. Indeed, let $X' \dto X$ be a bimeromorphic map from a compact K\"ahler manifold $X'$ and let $\tilde{X'} \to X$ be a resolution of $X' \dto X$ by successively blowing-up $X'$ along smooth subvarieties. Then $\tilde{X'}$ is a K\"ahler manifold and $\tilde{X'} \to S$, which is the composition of $\tilde{X'} \to X$ with $X \rightarrow S$, is an elliptic fibration bimeromorphic to $X$. So we may replace $X$ by $\tilde{X'}$ for instance.

Let $\omega \in H^2(X,\R)$ be a K\"ahler class on $X$. By density, there exists a class $\alpha\in H^2(X,\Q)$
(in general not of type $(1,1)$) such that $\alpha\cdot F\neq 0$ where $F$ is a general $f$-fibre. The class $\alpha$ defines a global section of $R^2f_*\Q_X$ and we can cancel the denominators in such a way that $\alpha$ defines a section of $R^2f_*\Z_X$ and thus a non-zero element $\bar \alpha \in H^0(S, R^2f_*\Z_X/\mathcal{V}_X)$. 
The third line of \eqref{eq:diagram} induces an exact sequence
$$
H^0(S, R^2f_*\Z_X/\mathcal{V}_X) \stackrel{\tau}{\rightarrow} H^0(S, \Z_S) \stackrel{\delta}{\rightarrow} H^1(S,\mathcal{Q}_H)
$$
and is straightforward to check that $\tau(\bar \alpha)=F\cdot \alpha$. It is then a positive multiple of the class $1\in H^0(S,\Z_S)$ and it follows that $\delta(1)$ is a torsion class in $H^1(S,\mathcal{Q}_H)$. A diagram chase in \eqref{eq:diagram} shows that $\delta(1)$ is the image of $\eta(f)$ in $H^1(S,\mathcal{Q}_H)$.
\end{proof}

We can now generalise Theorem \ref{theoremsmoothdeformation} (cf. also \cite[Proposition 7.4.2]{Nak02c}):

\begin{theorem}\label{th:c torsion general}
In the situation of Assumption \ref{assume}, let us also assume that $X$ is bimeromorphic to a compact K\"ahler manifold.
Suppose also that $\eta(f)$ is in the image of $H^1(S,\mathcal{J}(H)^\mathbb{W})$.
Denote by 
$$
\mathtt{c}: H^1(S,\mathcal{J}(H)^\mathbb{W}) \rightarrow H^2(S, j_* H)
$$
the morphism defined by the first line of the exact sequence \eqref{eq:diagram}.
Then the class $\mathtt{c}(\eta(f))$ is torsion in $H^2(S,j_*H)$.
\end{theorem}

\begin{proof} 
Let $\eta \in H^1(S, \cJ(H)^\W)$ be an element which maps to $\eta(f) \in H^1(S,\mathcal{J}(H)_{mer})$ and let $p^\eta: W \colonec \W^\eta \to S$ be the minimal locally Weierstrass fibration which represents $\eta \in H^1(S, \cJ(H)^\W)$. By~\cite[Lemma 8.1]{Lin16}, the following diagram commutes
\begin{equation*}
\begin{tikzcd}[cramped, row sep = 20, column sep = 20]
H^0(S,R^2p^\eta_*\bZ) \simeq H^0(S,\bZ) \arrow[r, "\ga"] \arrow[bend left=10, rr,  "d_2"]   & H^1(S, \cJ(H)^\W )    \arrow[r, "\mathtt{c}"] & H^2(S,R^1p^\eta_*\bZ) \simeq H^2(S,j_*H)  \\ 
\end{tikzcd}
\end{equation*}
where  $\mathtt{c}$ and $\ga$ are the connecting morphisms in the long exact sequences induced by~\eqref{ex-1} and~\eqref{ex-2} respectively. As $\ga(1) = \eta$, it suffices to prove the following lemma, which implies that $d_2 \otimes \bR = 0$. 
\end{proof}

\begin{lemma}
 $H^2(W, \bR) \to H^0(S,R^2p^\eta_*\bR)$ is surjective.
\end{lemma}
\begin{proof}
Since $W$ is normal and has at worst rational singularities, by~\cite[Injection (3)]{a21} we have an injection
$$H^{1,1}_{BC}(W) \hto H^2(W,\bR).$$
Assume to the contrary that $H^2(W, \bR) \to H^0(S,R^2p^\eta_*\bR)$ is not surjective, so in particular its restriction to $H^{1,1}_{BC}(W)$ is not surjective. Let $\tau : \tilde{W} \to W$ be a K\"ahler desingularization of $W$. By the projection formula, given an element $\go \in H^2(W, \bR)$, its image in $H^0(S,R^2p^\eta_*\bR) \simeq H^0(S,\bR) \simeq \bR$ equals $\int_F \tau^*\go$ where $F$ is a smooth fiber of $p^\eta \circ \tau: \tilde{W} \to S$. Let $n\colonec \dim W$. The non-surjectivity assumption implies that $\tau^*H^{1,1}_{BC}(W) \subset [F]^\perp$ where the orthogonal is with respect to the Poincar\'e duality pairing 
$$H^{n-1, n-1}(\tilde{W})_\bR \times H^{1,1}(\tilde{W})_\bR \to H^{n,n}(\tilde{W})_\bR \simeq \bR.$$
However since $\ker(\tau_*)^\perp \subset \tau^*H^{1,1}_{BC}(W)$ by~\cite[Lemma 3.3]{a21}, we deduce that $\tau_*[F] = 0$, which is not possible.
\end{proof}

\begin{remark}\label{rem:Kodaira versus us}
The last result gives a direct proof of a phenomenon which was \emph{observed} by Kodaira in
the case $\dim S=1$: he first proved that the cohomology group $H^2(S,j_*H)$ is finite if the \textsc{vhs} is not trivial. He then computed the first Betti number of an elliptic surface when $H$ is trivial and obtained  in \cite[Theorem 11.9]{Kod} that this quantity is even if $\mathtt{c}(\eta(f)) = 0$ and odd otherwise. \emph{A posteriori} we can conclude that an elliptic surface $f:X\to S$ (without multiple fibres) is K\"ahler if and only if $\mathtt{c}(\eta(f))$ is torsion in $H^2(S,j_*H)$. We will now prove that this equivalence also holds in our setting:
\end{remark}

\begin{proposition}\label{prop:criterion_kaehler_elliptic}
In the situation of Assumption \ref{assume}, suppose also that the base $S$ is a compact K\"ahler manifold. Assume that the Weierstra{\ss} fibration $\W \to S$ is minimal. Let $\eta\in H^1(S,\mathcal{J}(H)^\mathbb{W})$ be a class such that $\mathtt{c}(\eta)$ is torsion in $H^2(S,j_*H)$.
Then the total space $\mathbb{W}^\eta \rightarrow S$ is K\"ahler space.
\end{proposition}

\begin{proof}
Recall that the Weierstra\ss~ model $p_{\mathbb W}: \mathbb W \rightarrow S$ associated to $H$ is a projective morphism. 
Since $S$ is compact K\"ahler, the total space $\mathbb{W}$ is K\"ahler by Remark \ref{remarkrelplusbase}.
As in the case of smooth elliptic fibrations, the  Weierstra\ss~fibration comes equipped with a family of of elliptic fibrations (over $S$) $\mathcal{W}\To S\times H^1(S,\mathcal{L})$ parametrised by the vector space $H^1(S,\mathcal{L})$ such that
$$\eta(\mathcal{W}_t\to S)=\exp(t)\in H^1(S,\mathcal{J}(H)^\mathbb{W})$$
for any $t\in H^1(S,\mathcal{L})$. By Remark \ref{remarkpropertiesweierstrass} 
the complex spaces $\mathbb{W}^\eta$ have at most canonical, hence rational, singularities.
From \cite[Proposition 5]{Nam01} we know that any small flat deformation of compact K\"ahler space having rational singularities remains K\"ahler. Thus $\mathbb{W}^{\exp(t)}$ is K\"ahler for $t$ in a neighborhood $U$ of $0\in H^1(S,\mathcal{L})$. Now if $t$ is given in $H^1(S,\mathcal{L})$ let us consider a positive integer $m$ such $t/m\in U$. The multiplication-by-$m$ 
$$
\bm : \mathbb{W}^{\exp(t/m)}\To \mathbb{W}^{m\cdot \exp(t/m)}=\mathbb{W}^{\exp(t)}.
$$
is a finite morphism by Lemma~\ref{lem-multm}.
Since $\mathbb{W}^{\exp(t/m)}$ is K\"ahler,  \cite[Cor.3.2.2]{Var89} implies that $\mathbb{W}^{\exp(t)}$ is K\"ahler as well.

Since $\mathtt{c}(\eta)$ is torsion by assumption, there exists a positive integer $k$ and an element $t\in H^1(S,\mathcal{L})$ such that $k\cdot \eta=\exp(t)$. As the multiplication-by-$k$
$$\mathbb{W}^\eta\To \mathbb{W}^{k\cdot\eta}=\mathbb{W}^{\exp(t)}$$
is finite and the target is K\"ahler, we conclude that $\mathbb{W}^{\eta}$ is also K\"ahler (see for instance \cite[Proposition 1.3.1]{Var89}).
\end{proof}

\subsubsection{Hodge theory of Weierstra\ss~models}
The main purpose of this paragraph is to establish the following result.
\begin{theorem}\label{th:hodge_surjective}
In the situation of Assumption \ref{assume}, suppose that $S$ is a compact K\"ahler manifold. Then the canonical map
\begin{equation}\label{eq:Hodge_surjective}
H^1(S,j_*H_\R)\To H^1(S,\mathcal{L})
\end{equation}
is surjective.
\end{theorem}

When $S^*=S$ this is a straightforward consequence of the existence of a pure Hodge structure of weight 2 on the lattice $H^1(S,H)$ as constructed by Deligne (see \cite[Theorem 2.9]{Zuc79}).

Theorem~\ref{th:hodge_surjective} will serve as a crucial ingredient in the proof of Theorem~\ref{theoremapproximation}. More precisely, it will be the following corollary that we use in the proof.

\begin{corollary}\label{cor:densite points rationnels G}
In the situation of Assumption~\ref{assume}, let $G$ be a finite group acting $f$-equivariantly on $X$ and on $S$. Then the image of $H^1(S,j_*H_\Q)^G$ in $H^1(S,\mathcal{L})^G$ under the map $H^1(S,j_*H_\R)^G\To H^1(S,\mathcal{L})^G$ is dense.
\end{corollary}

\begin{proof}
As $G$ is a finite group and since~\eqref{eq:Hodge_surjective} is  surjective by Theorem~\ref{th:hodge_surjective}, its $G$-invariant part is also surjective. Corollary~\ref{cor:densite points rationnels G} thus follows from the density of $H^1(S,j_*H_\Q)^G$ in $H^1(S,j_*H_\R)^G$.
\end{proof}

\begin{remark}\label{rem:surjectivity_density}
Before giving the proof of Theorem \ref{th:hodge_surjective}, let us remark that this statement has a geometric counterpart. Let $\W \to S$ be the minimal Weierstra{\ss} fibration associated to the \textsc{vhs} $H$. As $S$ is assumed to be compact K\"ahler, by Remark~\ref{remarkrelplusbase}, Theorem~\ref{th:c torsion general}, and Proposition~\ref{prop:criterion_kaehler_elliptic} the total space $\W$ is K\"ahler. So each fiber of the tautological family $\mathcal{W}\To S\times H^1(S,\mathcal{L})\To H^1(S,\mathcal{L})$ containing $\W \to S$ as a fiber (i.e. the family constructed in Proposition~\ref{prop:deformation loc trivial} for $\eta = 0$) is also K\"ahler. Theorem \ref{th:hodge_surjective} is thus equivalent to the density of fibrations $\mathcal{W}_t\to S$ that are projective over $S$. Indeed such an elliptic fibration is projective over $S$ if and only if its cohomology class $\eta(\mathcal{W}_t\to S)$ is torsion in $H^1(S,\mathcal{J}(H)^\mathbb{W})$ (see \cite[Theorem 6.3.8]{Nak02c}). Using the first line of \eqref{eq:diagram}, the exact sequence
$$
H^1(S,j_*H)\To H^1(S,\mathcal{L})\To H^1(S,\mathcal{J}(H)^\mathbb{W})
$$
shows that this happens exactly when $t$ lies in the range of the map $H^1(S,j_*H_\Q)\To H^1(S,\mathcal{L})$. Now it is clear that density of projective elliptic fibrations is equivalent to surjectivity of the map (\ref{eq:Hodge_surjective}).
\end{remark}

To prove Theorem \ref{th:hodge_surjective}, we will establish the density of projective fibrations in the tautological family. The following criterion is reminiscent from Buchdahl's works \cite{Buc06,Buc08}:
 
\begin{lemma}\label{lem:Buchdahl_criterion}
Let $\pi:\mathcal{X}\To B$ be a smooth family of compact K\"ahler manifolds, and let
$\Phi: \mathcal{X}\To S\times B\To B$ be a fibration such that $\pi = pr_{B} \circ \Phi$.
Consider the following \textsc{vhs} over $B$:
$$\mathbb{V}:=R^2\pi_*\Q/H^2(S,\Q).$$
Let $b \in B$ be a point, and $[\omega]$ a K\"ahler class defined on $X:=\mathcal{X}_b$.
If the composition of the maps
$$
T_{B,b}\stackrel{\kappa\varsigma}{\To} H^1(X,T_X)\stackrel{\bullet\wedge [\omega]}{\To} H^2(X,\sO_X)\To \mathbb{V}_b^{0,2}
$$
is surjective, then the set of parameters $u\in B$ such that the morphism $\mathcal{X}_u\to S$ is projective is dense near $b$.
\end{lemma}
In the statement above, the first arrow is the Kodaira-Spencer map associated to $\pi$, and the second one is induced by the contraction with the class $\omega\in H^1(X,\Omega^1_X)$.

\begin{proof}
This is nothing but \cite[Proposition 17.20, p.410]{V02} applied to the \textsc{vhs} $\mathbb{V}$.
\end{proof}

The deformation families provided by Nakayama's theory are not smooth, so in order to apply the relative Buchdahl criterion
we have to pass to a smooth model. Koll\'ar's theory of strong resolutions \cite[Chapter 3]{Kol07} gives a resolution in families:

\begin{lemma}\label{lem:resolution_family}
Let $\holom{p_0}{\mathbb W_0}{S}$ be a fibration from a normal compact complex space onto a compact complex manifold $S$, and
$p: \sW \rightarrow S \times B$ be a locally trivial deformation of $(\mathbb W_0, p_0)$ (cf. Definition \ref{definitionlocallytrivial}). Then (up to replacing $B$ by a smaller open set) there exists a resolution of singularities 
$\holom{\mu}{\sX}{\sW}$ such that the family
$$
\pi := pr_{B} \circ p \circ \mu : \sX \rightarrow B
$$
is a family of compact complex manifolds, and for every $b \in B$ the map $\holom{\mu_b}{\mathcal X_b}{\mathbb W_b}$
is a resolution of singularities that is functorial.
\end{lemma}
\begin{proof}
By \cite[Theorem 3.35]{Kol07} (the analytic situation is dealt with in \cite{Wlo09}) there exists a functorial resolution $\holom{\mu}{\sX}{\sW}$. This resolution commutes with
smooth maps, so if $(U_i)_{i \in I}$ is a finite open cover\footnote{By Definition \ref{definitionlocallytrivial} these covers
exist up to replacing $B$ by a smaller open subset.} of $S$ 
such that $\fibre{p}{U_i \times B} \simeq \fibre{p_0}{U_i} \times B$, then the resolution is the product
of the functorial resolution of  $\fibre{p_0}{U_i}$ times the identity.
\end{proof}

We follow the convention of \cite[Appendix B, p.287]{Ser06}: given a morphism $\holom{f}{X}{S}$ of normal varieties, we denote by 
$T_{X/S}$ the dual of the sheaf of K\"ahler differentials $\Omega_{X/S}$. In particular $T_{X/S}$ is always a reflexive sheaf.

\begin{lemma}
Let \holom{p^\eta}{\mathbb W^\eta}{S} be a minimal local Weierstra\ss~fibration over a smooth base $S$.
Then we have 
\begin{equation} \label{relativeT}
T_{\mathbb{W}^\eta/S} \simeq (p^\eta)^* \cL.
\end{equation}

Moreover, let $\holom{\mu}{X}{\mathbb W^\eta}$ be a functorial resolution of singularities and set $f:=p^\eta \circ \mu$.
Then there exists a natural injection $\mu^* T_{\mathbb{W}^\eta/S} \rightarrow T_{X/S}$ inducing a map
\begin{equation} \label{pullH1}
H^1(\mathbb W^\eta, T_{\mathbb W^\eta/S})
\rightarrow
H^1(X, T_{X/S}).
\end{equation}
\end{lemma}

\begin{proof}
All the fibres of $p^\eta$ are reduced plane cubics, so there exists
a codimension two subset $Z \subset \mathbb W^\eta$ such that $p^\eta|_{\mathbb{W}^\eta \setminus Z}$ is a smooth fibration. On this
smooth locus we have by construction $T_{\mathbb{W}^\eta \setminus Z/S} \simeq (p^\eta|_{\mathbb{W}^\eta \setminus Z})^* \cL$. Since $T_{\mathbb{W}^\eta/S}$ and $p^* \cL$ are both reflexive and $\W^\eta$ is normal, the isomorphism extends to an isomorphism on $\mathbb W^\eta$.
This shows \eqref{relativeT}, in particular $T_{\mathbb{W}^\eta/S}$ is locally free.

Since the resolution $\mu$ is functorial, the direct image sheaf $\mu_* (T_X) \subset T_{\mathbb W^\eta}$ is reflexive \cite[Corollary 4.7]{GKK10}. Thus for any open subset $U \subset \mathbb W^\eta$, the restriction map
$$
\Gamma(\fibre{\mu}{U}, T_X) \rightarrow \Gamma(\fibre{\mu}{U} \setminus \mbox{Exc}(\mu), T_X)
$$
is surjective. Using the exact sequence
$$
0 \rightarrow T_{X/S} \rightarrow T_X \rightarrow f^* T_S
$$
and the fact that $f^* T_S$ is torsion-free, we obtain that
$$
\Gamma(\fibre{\mu}{U}, T_{X/S}) \rightarrow \Gamma(\fibre{\mu}{U} \setminus \mbox{Exc}(\mu), T_{X/S})
$$
is surjective. Thus $\mu_* (T_{X/S})$ is reflexive, and the natural map $\mu_* (T_{X/S}) \rightarrow T_{\mathbb{W}^\eta/S}$ is an
isomorphism. By applying the projection formula to the inverse $T_{\mathbb{W}^\eta/S} \rightarrow \mu_* (T_{X/S})$ we obtain an
injective morphism
$$
\mu^* T_{\mathbb{W}^\eta/S} \rightarrow T_{X/S}.
$$
Since $T_{\mathbb{W}^\eta/S}$ is locally free and $\mathbb{W}^\eta$ has rational singularities, we have an isomorphism
$$
H^1(\mathbb W^\eta, T_{\mathbb W^\eta/S}) \simeq H^1(X, \mu^* T_{\mathbb W^\eta/S})
$$
given by $\mu^*$. The statement follows by composing this isomorphism with the map 
$H^1(X, \mu^* T_{\mathbb W^\eta/S}) \rightarrow H^1(X, T_{X/S})$.
\end{proof}

Following \cite[Chapter 3.4.2]{Ser06} (cf. \cite{FK87} for a presentation in the analytic setting) 
we consider the functor of locally trivial deformations of $\holom{f}{X}{S}$ with fixed target $S$.
By \cite[Lemma 3.4.7.b) and Theorem 3.4.8]{Ser06} we have an injection
$$
H^1(X, T_{X/S}) \hookrightarrow D_{X/S},
$$
where $D_{X/S}$ is the tangent space of the semiuniversal deformation. Given a locally trivial deformation
$$
\xymatrix{
\mathcal{X} \ar[r]^\Phi \ar[rd]_\pi & S \times B \ar[d]^{pr_B}
 \\
& B 
}
$$
parametrised by a smooth base $B$, we have for every $b \in B$ a Kodaira-Spencer map
$$
\kappa \varsigma_{\Phi}: T_{B,b} \rightarrow D_{X/S}
$$
associating a tangent vector with the corresponding first-order deformation.

\begin{proof}[Proof of Theorem \ref{th:hodge_surjective}]
Let us first recall from Remark \ref{rem:surjectivity_density} that the surjectivity involved in the statement of Theorem \ref{th:hodge_surjective} is equivalent to the density of projective fibrations in the tautological Weierstra\ss~family
$$
\holom{p}{\sW}{S \times H^1(S, \mathcal L)}.
$$
Fix now $b \in H^1(S,\sL)$. 
By Lemma \ref{lem:resolution_family} there exists (up to replacing the base $H^1(S,\sL)$ by a neighbourhood of the point $b$) 
a simultaneous functorial resolution of the tautological family:
$$
\xymatrix{\mathcal{X} \ar[rr]^\mu \ar[rd]^\Phi \ar @/_/ [rdd]_\pi & &\mathcal{W}\ar[ld]_p \ar @/^/[ldd] \\
& S \times H^1(S,\sL)\ar[d]&\\ & H^1(S,\sL) &}
$$
In order to simplify the notation, we replace
$\holom{p_b}{\mathbb{W}^{\exp(b)}}{S}$ by $\holom{p_b}{\mathbb{W}_b}{S}$. We
have a commutative diagram
$$
\xymatrix{
\sX_b \ar[rr]^{\mu_b} \ar[rd]_{f_b} && \mathbb{W}_b \ar[ld]^{p_b} \\&S&
}
$$
where $\mu_b$ is a functorial resolution of singularities of $\mathbb{W}_b$. 

{\em Step 1. The Kodaira-Spencer map $\kappa \varsigma_{\Phi,b}$ is given by $f_b^*$.}
Using \eqref{relativeT} one shows easily that the  Kodaira-Spencer map $\kappa\varsigma_p$ for any point $b \in H^1(S, \mathcal L)$ identifies to the pull-back
$$
p_b^*: H^1(S, \sL) \rightarrow H^1(\mathbb{W}_b, p_b^* \sL) \simeq 
H^1(\mathbb{W}_b, T_{\mathbb{W}_b/S}).
$$ 
By functoriality of the Kodaira-Spencer map we have a factorisation
$$
\xymatrix{
H^1(S,\sL)
\ar[r]^{\kappa\varsigma_{\Phi, b}} 
\ar[rd]_{\kappa\varsigma_{p, b}}
& H^1(\sX_b ,T_{\sX_b/S})
\ar[d]
\\
& 
H^1(\mathbb{W}_b, T_{\mathbb{W}_b/S}).
}
$$
By \eqref{pullH1} the right column has 
an inverse $H^1(\mathbb{W}_b, T_{\mathbb{W}_b/S}) \rightarrow H^1(\sX_b ,T_{\sX_b/S})$ defined by $\mu_b^*$.
Since $\kappa\varsigma_{p,b}$ identifies to the pull-back $p_b^*$ we obtain that
$\kappa\varsigma_{\Phi,b}$ identifies to
$$
f_b^* : H^1(S, \sL) \rightarrow H^1(\sX_B, (f_b)^* \sL).
$$

{\em Step 2. Applying the relative Buchdahl criterion.}
We want to apply Lemma \ref{lem:Buchdahl_criterion} to the family $\sX\to S\times H^1(S,\sL)\to H^1(S,\sL)$. 
Fix a point $b \in H^1(S,\sL)$, and observe that 
$$
\mathbb{V}_b^{0,2} = H^2(\sX_b,\sO_{\sX_b})/H^2(S,\sO_S).
$$
We have to check that the composed map
\begin{equation}\label{eq:composed_map}
H^1(S,\sL)\stackrel{\kappa\varsigma_{\Phi,b}}{\To}H^1(\sX_b,T_{\sX_b})\stackrel{\bullet\wedge[\omega]}{\To} H^2(\sX_b,\sO_{\sX_b})\To H^2(\sX_b,\sO_{\sX_b})/H^2(S,\sO_S)
\end{equation}
is surjective. 

\noindent\textit{Claim: the quotient $H^2(\sX_b,\sO_{\sX_b})/H^2(S,\sO_S)$ is isomorphic to $H^1(S,\sL)$ and the projection $H^2(\sX_b,\sO_{\sX_b}) \To H^1(S,\sL)$ is induced by $(f_b)_*$.}

By Theorem \ref{thm-van} we have $R^1 (f_b)_* \sO_{\sX_b} \simeq \sL$ and $R^2 (f_b)_* \sO_{\sX_b}=0$.
Moreover, since $\sX_b$ is compact K\"ahler, the natural map $H^2(S, \sO_S) \rightarrow H^2(\sX_b,\sO_{\sX_b})$ is injective.
Thus the Leray spectral sequence yields an exact sequence
$$
0 \rightarrow H^2(S, \sO_S) \rightarrow H^2(\sX_b,\sO_{\sX_b}) \rightarrow H^1(S, R^1 (f_b)_* \sO_{\sX_b}) \simeq
H^1(S, \sL) \rightarrow 0.
$$
This shows the claim.

Combining Step 1 and the claim the composed map \eqref{eq:composed_map} is given by:
$$
H^1(S,\sL)\stackrel{f_b^*}{\To}H^1(\sX_b,T_{\sX_b})\stackrel{\bullet\wedge[\omega]}{\To} H^2(\sX_b,\sO_{\sX_b})
\stackrel{(f_b)_*}{\To} H^2(\sX_b,\sO_{\sX_b})/H^2(S,\sO_S).
$$
Using Dolbeault representatives
we see that
$$
(f_b)_*(\omega \wedge f_b^*(\alpha))= c_\omega \alpha
$$
where $c_\omega= \omega \cdot F$ with $F$ a general $f_b$-fibre.
Thus (\ref{eq:composed_map}) is even an isomorphism. 
\end{proof}


\section{Proofs of the main results}

\subsection{Proof of Theorem \ref{theoremapproximation} and Corollary~\ref{theoremmain}}

\begin{proof}[Proof of Theorem \ref{theoremapproximation}]
Since $a(X) = \dim X - 1$ and $X$ is K\"ahler, the algebraic reduction $X \dto S$ of $X$ is an almost holomorphic map onto a projective variety $S$ whose general fiber is an elliptic curve. Up to replacing $X$ by a bimeromorphic model, we can assume that $X \to S$ is an elliptic fibration such that both $X$ and $S$ are smooth, $X$ is K\"ahler, $S$ is projective, and the discriminant locus $\gD \subset S$ is a normal crossing divisor. By Proposition~\ref{prop:local meromorphic section}, there exists a finite Galois cover $r : \tilde{S} \to S$ such that the elliptic fibration $\tilde{X} \colonec X \times_S \tilde{S} \to \tilde{S}$ has local meromorphic sections at every point of $\tilde{S}$; let $G \colonec \Gal(\tilde{S} / S)$.

Let $\eta \in H^1(\tilde{S},\cJ(\tilde{H})_{mer})$ be the element associated to $\tilde{X} \to \tilde{S}$ where $\tilde{H} \colonec r^* H$ and let $V \colonec H^1(\tilde{S},\cL_{\tilde{H}/\tilde{S}})$.
The $G$-action on $\tilde{H}$ induces a $G$-action on $V$; let $V^G$ be the $G$-invariant part. 

\begin{lemma}\label{lem-loctriv}
There exist a family $\Pi : \wti{\cX} \to \tilde{S} \times V^G $ of elliptic fibrations over $\tilde{S}$ parameterized by $V^G$  together with  a $\Pi$-equivariant $G$-action on $\wti{\cX}$ such that 
\begin{enumerate}[i)]
\item the elliptic fibration $\wti{\cX}_\phi \to \tilde{S}$ parameterized by $\phi \in V^G$ corresponds to $\exp{\phi} + \eta \in H^1(\tilde{S},\cJ(\tilde{H})_{mer})$;
\item there exist a $G$-invariant open cover $\{U_i\}$ of $\tilde{S}$ and a minimal Weierstra{\ss} fibration $\pi : W \to \tilde{S}$ together with a $\pi$-equivariant $G$-action such that $\Pi^{-1}(U_i \times V^G) \simeq W_i \times V^G $ over $U_i$ where $W_i \colonec \pi^{-1}(U_i)$ (in particular $\Pi$ is locally trivial in the sense of Definition~\ref{definitionlocallytrivial});
\item moreover, for each $g \in G$ and each pair of indices $i$ and $j$ such that $g(U_i) = U_j$, the restriction $g : \Pi^{-1}(U_i \times V^G ) \xrightarrow{\sim} \Pi^{-1}(U_j \times V^G)$ of the $G$-action on $\wti{\cX}$ to $\Pi^{-1}(U_i \times V^G)$ is isomorphic to $(g_{|W_i} \times \Id) : W_i \times V^G  \xrightarrow{\sim} W_j \times V^G$.  
\end{enumerate}
As an immediate consequence of ii) and iii), the quotient $\wti{\cX}/G \to S \times V^G$ of $\wti{\cX} \to \tilde{S} \times V^G$ is a locally trivial family of elliptic fibrations over $S$ parameterized by $V^G$.
\end{lemma}

\begin{proof}
The inclusion $V^G \hto H^1(\tilde{S},\cL_{\tilde{H}/\tilde{S}})$
gives rise to an element 
$$\xi \in H^1(\tilde{S},\cL_{\tilde{H}/\tilde{S}}) \otimes H^0(V^G,\cO_{V^G}) \subset H^1(\tilde{S} \times V^G,\cL_{p^*\tilde{H}/\tilde{S} \times V^G})$$
where $p : \tilde{S} \times V^G \to \tilde{S}$ denotes the projection onto the first factor. Since $\exp{\xi} $ is sent to $0$ in  $H^1(\tilde{S} \times V^G,\cQ_{p^*\tilde{H}})$ and since the image of $\eta$ in $H^1(\tilde{S} ,\cQ_{\tilde{H}})$ is torsion by Lemma~\ref{lem:Kahler implies virtually weierstrass} because $\tilde{X}$ is K\"ahler, the image of the element 
$$\eta_0 \colonec \exp{\xi} + p^*\eta \in H^1(\tilde{S} \times V^G,\cJ(p^*\tilde{H})_{mer})$$ 
in $H^1(\tilde{S} \times V^G,\cQ_{p^*\tilde{H}})$ is also torsion. So there exists $k \in \bZ_{>0}$ such that 
$$k\eta_0 \in H^1(\tilde{S} \times V^G,\cJ(p^*\tilde{H})^{\W}).$$

Already by~\cite[Theorem 6.3.12]{Nak02c}, we know that $\eta_0$ corresponds to an elliptic fibration $\Pi : \wti{\cX} \to V^G \times \tilde{S}$ and by construction, the elliptic fibration $\wti{\cX}_\phi \to \tilde{S}$ parameterized by $\phi \in V^G$ corresponds to $\exp{\phi} + \eta \in H^1(\tilde{S},\cJ(\tilde{H})_{mer})$. To show the existence of such a family which satisfies other properties listed in Lemma~\ref{lem-loctriv}, let us recall how in our case $\Pi$ is constructed following the proof of~\cite[Theorem 6.3.12]{Nak02c}.

Let $\Pi' : \wti{\cX}' \to \tilde{S} \times V^G$ be the minimal locally Weierstrass model which corresponds to $k\eta_0$. This is a subfamily of the one constructed in Proposition~\ref{prop:deformation loc trivial} for $k \eta \in H^1(\tilde{S},\cJ(\tilde{H})^{\W})$. By construction, there exists a $G$-invariant open cover $\{U_i\}$ of $\tilde{S}$ such that $\Pi'^{-1}(U_i \times V^G) \simeq \pi^{-1}(U_i) \times V^G$ over $V^G$ where $\pi : \tilde \W \to \tilde{S}$ is the minimal  Weierstrass model associated to $\tilde{H}$ together with a $\pi$-equivariant $G$-action. Furthermore, $\Pi'$ satisfies Property $iii)$, namely for each $g \in G$ and each pair of indices $i$ and $j$ such that $g(U_i) = U_j$, the restriction $g : \Pi'^{-1}( U_i \times V^G) \xrightarrow{\sim} \Pi'^{-1}(U_j \times V^G)$ of the $G$-action on $\wti{\cX}'$ to $\Pi'^{-1}(U_i \times V^G)$ is isomorphic to $(g_{|\tilde \W_i} \times \Id) :  \tilde \W_i \times V^G \xrightarrow{\sim}  \tilde \W_j \times V^G$ where $\tilde \W_i \colonec \pi^{-1}(U_i)$. 

Let $(\eta_{ij})$ be a 1-cocycle representing $\eta_0$ with respect to the open cover $\{U_i \times V^G\}$, which induces bimeromorphic maps $\eta_{ij} : \tilde \W_{ij} \times V^G \dto \tilde \W_{ji} \times V^G$ where $\tilde \W_{ij} \colonec \tilde \W_i \cap \tilde \W_j$.  As  $\eta_0$ lies in the $G$-invariant part of $H^1(\tilde{S} \times V^G,\cJ(p^*\tilde{H})_{mer})$, we can assume that $(\eta_{ij})$ is $G$-invariant. Let $\bk : \tilde \W \to \tilde \W$ be the multiplication-by-$k$ map, which is finite by Lemma~\ref{lem-multm}. Let $\bk_{ij} : \tilde \W_{ij} \to \tilde \W_{ij}$ be the restriction of $\bk$ to $\tilde \W_{ij}$ and let $(\bk_{ij} \times \Id) : \tilde \W_{ij} \times V^G \to \tilde \W_{ij} \times V^G \subset \wti{\cX}'_i$ be its product with $\Id : V^G \to V^G$. As $\tilde \W_{ij} \times V^G$ is normal, $(\bk_{ij} \times \Id) = (\bk_{ji} \times \Id)\circ \eta_{ij}$, and $\bk_{ij} \times \Id$ is finite, the map $\eta_{ij}$ is holomorphic. Using the cocycle of biholomorphic maps $(\eta_{ij})$, we can glue the  $ \tilde \W_i\times V^G$'s together to form an elliptic fibration $\Pi : \wti{\cX} \to \tilde{S} \times V^G$ which represents $\eta$. By construction $\Pi^{-1}(U_i \times V^G) \simeq \tilde \W_i \times V^G$ over $V^G$, which shows $ii)$. The $G$-invariance of $(\eta_{ij})$ gives rise to a $\Pi$-equivariant $G$-action on $\wti{\cX}$. Since $\Pi'$  satisfies Property $iii)$, it follows that $\Pi$ also satisfies $iii)$.
\end{proof}

There is a dense subset $\gS_0$ of $V^G$ parameterizing fibers of $\wti{\cX} \to V^G$ which are algebraic. Indeed, by the long exact sequence coming from~\eqref{exseq-JW} and Corollary~\ref{cor:densite points rationnels G}, the pre-images of torsion points of $H^1(\tilde{S},\cJ(\tilde{H})^\W)$ under the restriction of the exponential map $H^1(\tilde{S},\cL_{\tilde{H}/\tilde{S}})^G \to H^1(\tilde{S},\cJ(\tilde{H})^\W)$ form a dense subset in $H^1(\tilde{S},\cL_{\tilde{H}/\tilde{S}})^G$. As $\tilde{X}$ is K\"ahler, by Lemma~\ref{lem:Kahler implies virtually weierstrass} there exists $m \in \bZ_{>0}$ such that $m\eta \in H^1(\tilde{S},\cJ(\tilde{H})^\W)$. By Theorem~\ref{th:c torsion general}, we can further assume that up to replacing $m$ by a larger multiple, $m\eta = \exp(\phi_0)$ for some $\phi_0 \in V^G$. It follows that 
$$\gS \colonec \left\{ \phi \in V^G \mid \exp(\phi_0 + m\phi) \text{ is torsion} \right\} = \frac{1}{m}(\gS_0 - \phi_0)$$ 
is a dense subset of $V^G$. Since $\phi \in \gS$ implies that $\eta + \exp(\phi) \in H^1(\tilde{S},\cJ(\tilde{H})_{mer})$ is torsion, fibers of $\wti{\cX} \to V^G$ parameterized by $\gS$ are algebraic by~\cite[Proposition 5.5.4]{Nak02c}.

By Lemma~\ref{lem-loctriv}, the quotient $\wti{\cX}/G \to S \times V^G$ is a locally trivial family of elliptic fibrations over $S$ parameterized by $V^G$.
 Let $\cX$ be the functorial desingularization of $\wti{\cX}/G$. The family $\cX \to V^G$ is an algebraic approximation of the central fiber, which is smooth by~Lemma \ref{lem:resolution_family}, and is  bimeromorphic to $X$.
\end{proof}

\begin{remark}\label{rem-locW}
We notice that Lemma~\ref{lem-loctriv} (or more directly the same construction method of $\Pi$) implies in particular that if $f : X \to B$ is an elliptic fibration satisfying Assumption~\ref{assume} and whose total space $X$ is compact K\"ahler, then $f$ has a (unique) minimal locally Weierstra{\ss} model. Note also that despite the fact that $f$ is bimeromorphic to a locally Weierstra{\ss} fibration, in general only a multiple of $\eta(f)$ is in the image of $H^1(S,\cJ(H)^\W)$. 
\end{remark}

\begin{proof}[Proof of Corollary~\ref{theoremmain}]
By Theorem~\ref{theoremapproximation}, $X$ is bimeromorphic to a compact K\"ahler manifold $X'$ which has an algebraic approximation. In particular, $\pi_1(X')$ is projective. Since the fundamental group of a compact K\"ahler manifold is invariant under bimeromorphic transformations, we have $\pi_1(X) = \pi_1(X')$.
\end{proof}

\subsection{Proof of Theorem \ref{theorempione}}

Let $X$ be a compact K\"ahler manifold, and consider $\holom{g}{X}{Y}$ a fibration onto
a compact K\"ahler manifold $Y$. If $F$ is a general fibre, denote by $\pi_1(F)_X$ the image
of the morphism $\pi_1(F) \rightarrow \pi_1(X)$. Up to blowing up $X$ and $Y$ we can
suppose that the fibration $g$ is neat in the sense of \cite[Definition 1.2]{Cam04}.
By \cite[Corollary 11.9]{Cam11} we then have an exact sequence
\begin{equation} \label{sequencefibrations}
1 \rightarrow \pi_1(F)_X \rightarrow \pi_1(X) \rightarrow \pi_1(Y, \Delta) \rightarrow 1.
\end{equation}
where $\Delta$ is the orbifold divisor defined in \eqref{orbifoldbase}.

Let us also recall that by \cite{Cam94,Kol93} every compact K\"ahler manifold $X$ admits a (unique up to bimeromorphic equivalence of fibrations) almost 
holomorphic fibration 
$$
\holom{g}{X}{\Gamma(X)}
$$
with the following property: let $Z$ be a subspace with normalisation $Z' \rightarrow Z$ passing through a very general point $x\in X$. Then $Z$ is contained in the fibre through $x$ if and only if the natural map $\pi_1(Z') \rightarrow \pi_1(X)$ has finite image. This fibration is called the $\Gamma$-reduction of $X$ (Shafarevich map in the terminology of \cite{Kol93}).
Up to replacing $\gamma$ by some neat holomorphic model we thus obtain a fibration such that
$\pi_1(F)_X$ is finite and such that the dimension of the base is minimal among all fibrations
with this property. We call $\gamma \dim(X) := \dim (\Gamma(X))$ the $\gamma$-dimension of $X$.

In geometric situations it is often necessary to replace $X$ by some \'etale cover. It is easily seen that the situation can be made equivariant under the Galois group of the cover.

\begin{lemma}\label{lem:equivariance_shafarevich_map}
Let $X$ be a compact K\"ahler manifold acted upon by a finite group $G$. Then there exists a proper modification $\mu:\tilde{X}\to X$ and a holomorphic map $g:\tilde{X}\to Y$ such that:
\begin{enumerate}[$(i)$]
\item $\tilde{X}$ and $Y$ are compact K\"ahler manifolds.
\item $G$ acts on $\tilde{X}$ and $Y$.
\item $\mu$ and $g$ are $G$-equivariant for these actions.
\item $g$ is a neat model \cite[Definition 1.2]{Cam04} of the $\Gamma$-reduction of $X$.
\end{enumerate}
\end{lemma}

\begin{proof}
Let us consider $S$ the (normalisation of the) irreducible component of the cycle space $\mathcal{C}(X)$ which parametrizes the fibres of the $\Gamma$-reduction. 
By uniqueness of the latter, the group $G$ acts on $S$ and the natural meromorphic map $X\dashrightarrow S$ is $G$-equivariant. Now it is enough to perform $G$-equivariant resolution of singularities for $S$ and $G$-equivariant blow-ups on $X$ in order to make the latter map holomorphic, neat and $G$-equivariant.
\end{proof}

The following remark allows to control the extensions appearing in these covers.

\begin{remark}\label{rem:group_extension}
Let $1\to K\to H \to G \to 1$ be an exact sequence of groups. It is well known that this extension determines a morphism $\varphi_H:G \to \mathrm{Out}(K)$ to the group of outer automorphisms of $K$ (induced by the conjugation in $H$). To recover the extension, it is needed to prescribe an additional information: the class $c_H\in H^2(G, Z(K))$ (see \cite[Chapter IV, \S 6]{Br} for details).

On the reverse direction, when $G$ is a finite group acting by homeomorphisms on a topological space $Z$, we also have an induced morphism $\varphi_Z:G\to \mathrm{Out}(\pi_1(Z))$. In this case, there is an extension
$$
1\To \pi_1(Z)\To H\To G\To 1
$$
induced by the action of $G$ on $Z$. It can be explicitly constructed in the following way. 
By \cite{Ser58} there exists a projective simply connected manifold  
$P$ on which $G$ acts freely and we can look at the natural projection $Z\times P\to (Z\times P)/G$. 
This is a finite \'etale cover of Galois group $G$ and the homotopy exact sequence is:
\begin{equation}\label{eq:group action and extension}
1\To \pi_1(Z)\simeq\pi_1(Z\times P)\To \pi_1((Z\times P)/G)\To G \To 1.
\end{equation}
This is the sought exact sequence. It is stated in \cite[Lemma 3.9]{Cla16} that the extension (\ref{eq:group action and extension}) does not depend on $P$ and that this extension is the usual one if $G$ acts freely on $Z$ (\emph{i.e.} the one corresponding to $Z\to Z/G$).
\end{remark}

Using the construction above we obtain the following technical result.

\begin{lemma}\label{lem:basic-lemma}
Let $H:=\pi_1(X)$ be a K\"ahler group and $K \triangleleft H$ be finite index normal subgroup 
with quotient $G:=H/K$. Let us denote by $\tilde{X}$ the corresponding \'etale cover of $X$. Assume now that there exists a continuous map $g:\tilde{X}\to Z$ to a projective manifold $Z$ which is $G$-equivariant (so that $G$ acts on $Z$) and which induces an isomorphism at the level of fundamental groups. Then $H$ is a projective group.
\end{lemma}
\begin{proof}
The $G$-equivariant map
$$\tilde{X}\times P\stackrel{g}{\To} Z\times P$$
induces an isomorphism on the fundamental group and, using Remark \ref{rem:group_extension}, we infer that
$$\pi_1(X)\simeq \pi_1((Z\times P)/G)$$
is a projective group.
\end{proof}

Let us recall that a group $G$ is virtually torsion-free if there exist a subgroup $H \subset G$
of finite index that is torsion free.

\begin{lemma}\label{lem:virtual torsion freeness}
Let $X$ be a compact K\"ahler manifold admitting a fibration onto a curve $f:X\to C$ such that 
the fundamental group $\pi_1(F)$ of a general fibre $F$  is abelian. Then the group $\pi_1(X)$ is  virtually torsion-free. 
\end{lemma}
\begin{proof}
Applying \cite[Appendix C]{Cam98} we can take a finite \'etale cover such that $\pi_1(F)_X$ coincides with $K:=\ker(f_*:\pi_1(X)\to \pi_1(C))$ and the latter is thus finitely generated.
If $C \simeq \PP^1$ we obtain that $\pi_1(X) \simeq \pi_1(F)_X$ is abelian, so virtually torsion-free.
We can thus suppose that $g(C) \geq 1$.
By \cite[Theorem 5.1]{Ara11} the cohomology class $e\in H^2(\pi_1(C),K)$ 
corresponding to the extension
$$
1\To K\To \pi_1(X)\To \pi_1(C)\To 1
$$
is torsion. Then so is the class $e'\in H^2(\pi_1(C),K/K_{tor})$ corresponding to the extension
$$1\To K/K_{tor}\To \pi_1(X)/K_{tor}\To \pi_1(C)\To 1.$$
Arguing as in \cite[\S 2.1]{CCE2} we can assume that the latter cohomology class $e'$ vanishes (up to replacing $\pi_1(C)$ with a finite index subgroup). Using the following piece of long exact sequence of cohomology of $\pi_1(C)$-modules
$$\cdots\To H^2(\pi_1(C),K_{tor})\To H^2(\pi_1(C),K)\To H^2(\pi_1(C),K/K_{tor})\To\cdots$$
we see the cohomology class $e$ comes from $H^2(\pi_1(C),K_{tor})$. It is then easily observed\footnote{If $A$ is any finite $\pi_1(C)$-module then there is a finite index subgroup $\pi_1(C')$ of $\pi_1(C)$ such that the whole of the cohomology group $H^2(\pi_1(C),A)$ vanishes when restricted to $\pi_1(C')$. This is a consequence of the fact that a curve of positive genus admits finite \'etale covers of any given degree.} that this class is annihilated when restricted to a finite index subgroup of $\pi_1(C)$. This means that the following exact sequence of groups
$$1\To K_{tor}\To \pi_1(X)\To \pi_1(X)/K_{tor}\To 1$$
splits when restricted to a finite index subgroup and it proves that $\pi_1(X)$ is virtually torsion-free since $\pi_1(X)/K_{tor}$ is.
\end{proof}

Now we give some criteria to decide whether a K\"ahler group is projective.
\begin{lemma}\label{lem:gdim=1}
Let $X$ be a compact K\"ahler manifold having $\gamma\dim(X)\le1$. Its fundamental group is then projective.
\end{lemma}
\begin{proof}
If $\gamma\dim(X)=0$ its fundamental group is finite so \cite{Ser58} applies. If $\gamma\dim(X)=1$ we know from \cite[Th\'eor\`eme 1.2]{Cla10} (which is just a rephrase of \cite{Siu87}) that there exists a finite \'etale Galois cover $\pi:\tilde{X}\to X$ with
group $G$ such that the $\Gamma$-reduction of $\tilde X$ is a fibration $\tilde g: \tilde{X} \to C$
onto a curve inducing an isomorphism $\pi_1(X) \simeq \pi_1(C)$. 
It is also $G$-equivariant according to Lemma \ref{lem:equivariance_shafarevich_map}. We conclude by Lemma \ref{lem:basic-lemma}.
\end{proof}

\begin{lemma}\label{lem:gdim=2}
Let $X$ be a compact K\"ahler manifold such that $\gamma\dim(X) = 2$. If $\pi_1(X)$ is virtually torsion-free it is a projective group.
\end{lemma}
\begin{proof}
Let $K \triangleleft \pi_1(X)$ be a finite index subgroup  which is normal and torsion-free,
and set $G := \pi_1(X)/K$. Applying Lemma \ref{lem:equivariance_shafarevich_map} to the finite \'etale cover corresponding to $G$, we know that we can find $\tilde{X}\to X$ which is a composition of a finite \'etale cover and a modification such that the $\gamma$-reduction $g:\tilde{X}\to Y$ is neat, $Y$ is smooth K\"ahler surface and $g$ is equivariant for the natural actions of $G$ on $\tilde{X}$ and $Y$.
Consider now the exact sequence \eqref{sequencefibrations}: the group
$\pi_1(\tilde{X})$ is torsion-free and $\pi_1(F)_X$ is finite, so we have $\pi_1(F)_X=1$.
Thus $\pi_1(\tilde{X})\simeq \pi_1(Y,\Delta^*(g))$ is torsion-free, by Remark \ref{remarkp1orbifold}
this implies that 
$$
\pi_1(\tilde{X}) \simeq \pi_1(Y,\Delta^*(g)) \simeq \pi_1(Y).
$$
We can now argue according to the algebraic dimension of $Y$.
\begin{enumerate}
\item If $a(Y)=0$ then by the classification of surfaces $\pi_1(Y)$ is abelian.
Thus $\pi_1(X)$ is virtually abelian and \cite[Theorem 1.4]{BR11} applies.
\item If $a(Y)=1$ then the algebraic reduction $Y \rightarrow C$ is an elliptic fibration over a curve $C$.
Since the algebraic reduction is unique, it is $G$-equivariant.
By \cite[Theorems 14.1-3-5]{Kod} we know that there exists 
a $G$-equivariant deformation of $Y$ to an algebraic elliptic surface, so we can again conclude by Lemma \ref{lem:basic-lemma}.
\item If $a(Y)=2$ the surface $Y$ is projective and Lemma \ref{lem:basic-lemma} applies.
\end{enumerate}
\end{proof}

\begin{remark*}
Although Lemma \ref{lem:gdim=1} and \ref{lem:gdim=2} are stated in a very similar manner, they are of different nature: the former is a group theoretic statement  whereas the latter is not. Indeed as a consequence of \cite{Siu87} it is known that the property $\gamma d(X)=1$ is equivalent to having a fundamental group commensurable with the fundamental group of a curve; this property does thus depend only on the fundamental group. In general it is however possible to realize a given K\"ahler group as the fundamental group of several manifolds having different $\gamma$-dimensions.
\end{remark*}

\begin{proof}[Proof of Theorem \ref{theorempione}]
We argue according to the algebraic dimension of $X$, the case $a(X)=3$ being trivial
since a K\"ahler Moishezon manifold is projective.
\begin{enumerate}
\item If $a(X)=0$ then $X$ is special in the sense of Campana. 
Thus the fundamental group is virtually abelian \cite[Theorem 1.1]{CC14} 
and thus projective \cite[Theorem 1.4]{BR11}.
\item If $a(X)=1$, we replace $X$ by some blowup such that the algebraic reduction
is a holomorphic fibration $f: X \rightarrow C$ onto a curve. 
By \cite{CP00} the general fibre
of $F$ is bimeromorphic to a K3 surface, torus or ruled surface over an elliptic curve,
so its fundamental group is abelian. By Lemma \ref{lem:virtual torsion freeness}
the group $\pi_1(X)$ is virtually torsion free. If $\gamma \dim(X)\le 2$
we can thus apply Lemma \ref{lem:gdim=1} and Lemma \ref{lem:gdim=2}.
If $\gamma \dim(X)=3$ it is shown in \cite[Theorem 1]{CZ} that up to bimeromorphic transformations and \'etale cover $f$ is a smooth morphism. Thus we can apply \cite[Corollary 1.2]{Cla16}. 
\item If $a(X)=2$ the algebraic reduction makes $X$ into an elliptic fibre space over a projective surface and we can apply Corollary~\ref{theoremmain}.
\end{enumerate}
\end{proof}

\begin{appendix}
\section{Elliptic surfaces}

\begin{proposition} \label{propositionbeauvillekaehler}
Let $S$ be a non-algebraic compact K\"ahler surface that admits
an elliptic fibration $\holom{f}{S}{\PP^1}$. Then $f$ has at least three singular fibres.
\end{proposition}

\begin{proof}
We can suppose without loss of generality that $f$ is relatively minimal. Thus we know by the canonical bundle formula \cite{Kod} that
$$
K_S \simeq f^* (K_{\PP^1} + M + \sum_{c \in \PP^1} m_c S_c)
$$
where $M$ is the modular part defined by the $j$-function 
and $\sum_{c \in \PP^1} m_c S_c$ the discriminant divisor. Recall that any non-algebraic K\"ahler surface has a pseudoeffective canonical bundle, so $K_S$ is nef.

Suppose first that $f$ is isotrivial, i.e. we have $M \equiv 0$. Then \cite[Chapter V, Table 6]{BHPV04}
shows that the singular fibres are either multiples of smooth elliptic curves or of type $I_0^*$.
For a multiple fibre we have $m_c \leq \frac{1}{2}$ and for a fibre of type $I_0^*$ we have
$m_c = \frac{1}{2}$.  Since $K_{\PP^1} \simeq \sO_{\PP^1}(-2)$ we see that there are at least $4$ singular fibres. 

Suppose now that $f$ is not isotrivial. Then we can use the argument from \cite[Proposition 1]{Bea81}: let $C^\star \subset \PP^1$ be the maximal open set over which $f$ is smooth.
The $j$-function defines a non-constant holomorphic map $\tilde C^\star \rightarrow \mathbb H$ from the universal cover $\tilde C^\star \rightarrow C^\star$ to the upper half plane $\mathbb H$. In particular $\tilde C^\star$ is not $\C$ or $\PP^1$, hence  $\PP^1 \setminus C^\star$
has at least three points. 
\end{proof}

\end{appendix}


\begin{thebibliography}{BHPVdV04}

\bibitem[Ara11]{Ara11}
Donu Arapura.
\newblock Homomorphisms between {K}\"ahler groups.
\newblock In {\em Topology of algebraic varieties and singularities}, volume
  538 of {\em Contemp. Math.}, pages 95--111. Amer. Math. Soc., Providence, RI,
  2011.

\bibitem[BC15]{BC15}
Yohan Brunebarbe and Fr{\'e}d{\'e}ric Campana.
\newblock Fundamental group and pluridifferentials on compact {K}{\"a}hler
  manifolds.
\newblock {\em arXiv preprint}, 1510.07922, 2015.

\bibitem[Bea81]{Bea81}
Arnaud Beauville.
\newblock Le nombre minimum de fibres singuli{\`e}res d'une courbe stable sur
  {$\mathbb P^1$}.
\newblock In {\em S\'eminaire sur les {P}inceaux de {C}ourbes de {G}enre au
  {M}oins {D}eux}, volume~86 of {\em Ast\'erisque}, pages 97--108. Soci\'et\'e
  Math\'ematique de France, Paris, 1981.

\bibitem[BHPVdV04]{BHPV04}
Wolf~P. Barth, Klaus Hulek, Chris A.~M. Peters, and Antonius Van~de Ven.
\newblock {\em Compact complex surfaces}, volume~4 of {\em Ergebnisse der
  Mathematik und ihrer Grenzgebiete. 3. Folge.}
\newblock Springer-Verlag, Berlin, second edition, 2004.

\bibitem[Bin83]{Bin83}
J{\"u}rgen Bingener.
\newblock On deformations of {K}\"ahler spaces. {I}.
\newblock {\em Math. Z.}, 182(4):505--535, 1983.

\bibitem[BR11]{BR11}
Oliver Baues and Johannes Riesterer.
\newblock Virtually abelian {K}\"ahler and projective groups.
\newblock {\em Abh. Math. Semin. Univ. Hambg.}, 81(2):191--213, 2011.

\bibitem[Bro82]{Br}
Kenneth~S. Brown.
\newblock {\em Cohomology of groups}, volume~87 of {\em Graduate Texts in
  Mathematics}.
\newblock Springer-Verlag, New York-Berlin, 1982.

\bibitem[Buc06]{Buc06}
Nicholas Buchdahl.
\newblock Algebraic deformations of compact {K}\"ahler surfaces.
\newblock {\em Math. Z.}, 253(3):453--459, 2006.

\bibitem[Buc08]{Buc08}
Nicholas Buchdahl.
\newblock Algebraic deformations of compact {K}\"ahler surfaces. {II}.
\newblock {\em Math. Z.}, 258(3):493--498, 2008.

\bibitem[Cam94]{Cam94}
Fr{\'e}d{\'e}ric Campana.
\newblock Remarques sur le rev\^etement universel des vari\'et\'es
  k\"ahl\'eriennes compactes.
\newblock {\em Bull. Soc. Math. France}, 122(2):255--284, 1994.

\bibitem[Cam98]{Cam98}
Fr{\'e}d{\'e}ric Campana.
\newblock Negativity of compact curves in infinite covers of projective
  surfaces.
\newblock {\em J. Algebraic Geom.}, 7(4):673--693, 1998.

\bibitem[Cam04]{Cam04}
Fr{\'e}d{\'e}ric Campana.
\newblock Orbifolds, special varieties and classification theory.
\newblock {\em Ann. Inst. Fourier (Grenoble)}, 54(3):499--630, 2004.

\bibitem[Cam11]{Cam11}
Fr{\'e}d{\'e}ric Campana.
\newblock Orbifoldes g\'eom\'etriques sp\'eciales et classification
  bim\'eromorphe des vari\'et\'es k\"ahl\'eriennes compactes.
\newblock {\em J. Inst. Math. Jussieu}, 10(4):809--934, 2011.

\bibitem[CC14]{CC14}
Fr{\'e}deric Campana and Beno{\^{\i}}t Claudon.
\newblock Abelianity conjecture for special compact {K}\"ahler 3-folds.
\newblock {\em Proc. Edinb. Math. Soc. (2)}, 57(1):55--78, 2014.

\bibitem[CCE14]{CCE2}
Fr{\'e}deric Campana, Beno{\^{\i}}t Claudon, and Philippe Eyssidieux.
\newblock Repr\'esentations lin\'eaires des groups k\"ahl\'eriens et de leurs
  analogues projectifs.
\newblock {\em J. \'Ec. polytech. Math.}, 1:331--342, 2014.

\bibitem[CCE15]{CCE1}
Fr{\'e}deric Campana, Beno{\^{\i}}t Claudon, and Philippe Eyssidieux.
\newblock Repr\'esentations lin\'eaires des groupes k\"ahl\'eriens :
  factorisations et conjecture de {S}hafarevich lin\'eaire.
\newblock {\em Compos. Math.}, 151(2):351--376, 2015.

\bibitem[Cla10]{Cla10}
Beno{\^{\i}}t Claudon.
\newblock Invariance de la {$\Gamma$}-dimension pour certaines familles
  k\"ahl\'eriennes de dimension 3.
\newblock {\em Math. Z.}, 266(2):265--284, 2010.

\bibitem[Cla16]{Cla16}
Beno{\^i}t Claudon.
\newblock Smooth families of tori and linear {K}{\"a}hler groups.
\newblock To appear in \emph{Ann. Fac. Sc. Toulouse Math.}, arXiv:1604.03367,
  2016.

\bibitem[CP00]{CP00}
Fr{\'e}d{\'e}ric Campana and Thomas Peternell.
\newblock Complex threefolds with non-trivial holomorphic {$2$}-forms.
\newblock {\em J. Algebraic Geom.}, 9(2):223--264, 2000.


\bibitem[CZ05]{CZ}
F.~Campana and Q.~Zhang.
\newblock Compact {K}\"ahler threefolds of {$\pi_1$}-general type.
\newblock In {\em Recent progress in arithmetic and algebraic geometry}, volume
  386 of {\em Contemp. Math.}, pages 1--12. Amer. Math. Soc., Providence, RI,
  2005.

\bibitem[Dem85]{Dem85}
Jean-Pierre Demailly.
\newblock Mesures de {M}onge-{A}mp\`ere et caract\'erisation g\'eom\'etrique
  des vari\'et\'es alg\'ebriques affines.
\newblock {\em M\'em. Soc. Math. France (N.S.)}, 19:124, 1985.

\bibitem[FK87]{FK87}
Hubert Flenner and Siegmund Kosarew.
\newblock On locally trivial deformations.
\newblock {\em Publ. Res. Inst. Math. Sci.}, 23(4):627--665, 1987.

\bibitem[Fuj79]{Fujiki78}
Akira Fujiki.
\newblock Closedness of the {D}ouady spaces of compact {K}\"ahler spaces.
\newblock {\em Publ. Res. Inst. Math. Sci.}, 14(1):1--52, 1978/79.

\bibitem[GKK10]{GKK10}
Daniel Greb, Stefan Kebekus, and S\'andor~J. Kov\'acs.
\newblock Extension theorems for differential forms and {B}ogomolov-{S}ommese
  vanishing on log canonical varieties.
\newblock {\em Compos. Math.}, 146(1):193--219, 2010.

\bibitem[Gra62]{Gra62}
Hans Grauert.
\newblock \"{U}ber {M}odifikationen und exzeptionelle analytische {M}engen.
\newblock {\em Math. Ann.}, 146:331--368, 1962.

\bibitem[GR84]{GR84}
Hans Grauer and Reinhold Remmert.
\newblock Coherent analytic sheaves.
\newblock {Grundlehren der mathematischen Wissenschaften}, vol. 265, 1984.


\bibitem[Gra16]{Gra16}
Patrick Graf.
\newblock Algebraic approximation of {K}{\"a}hler threefolds of {K}odaira
  dimension zero.
\newblock {\em To appear in Math. Annalen}, 2016.

\bibitem[Har77]{Har77}
Robin Hartshorne.
\newblock {\em Algebraic geometry}.
\newblock Springer-Verlag, New York, 1977.
\newblock Graduate Texts in Mathematics, No. 52.

\bibitem[HP16]{a21}
Andreas H\"oring and Thomas Peternell.
\newblock Minimal models for {K}\"ahler threefolds.
\newblock {\em Invent. Math.}, 203(1):217--264, 2016.

\bibitem[Ive82]{Ive}
Birger Iversen.
\newblock {\em Cohomology of sheaves}, Universitext.
\newblock Springer-Verlag, New York-Berlin, 1986.

\bibitem[Kod60]{Kod60}
K.~Kodaira.
\newblock On compact complex analytic surfaces. {I}.
\newblock {\em Ann. of Math. (2)}, 71:111--152, 1960.

\bibitem[Kod63]{Kod}
K.~Kodaira.
\newblock On compact analytic surfaces. {II}, {III}.
\newblock {\em Ann. of Math. (2) 77 (1963), 563--626; ibid.}, 78:1--40, 1963.

\bibitem[Kol93]{Kol93}
J{\'a}nos Koll{\'a}r.
\newblock Shafarevich maps and plurigenera of algebraic varieties.
\newblock {\em Invent. Math.}, 113(1):177--215, 1993.

\bibitem[Kol07]{Kol07}
J\'anos Koll\'ar.
\newblock {\em Lectures on resolution of singularities}, volume 166 of {\em
  Annals of Mathematics Studies}.
\newblock Princeton University Press, Princeton, NJ, 2007.

\bibitem[Laz04]{Laz04a}
Robert Lazarsfeld.
\newblock {\em Positivity in algebraic geometry. {I}}, volume~48 of {\em
  Ergebnisse der Mathematik und ihrer Grenzgebiete.}
\newblock Springer-Verlag, Berlin, 2004.
\newblock Classical setting: line bundles and linear series.

\bibitem[Lin16]{Lin16}
Hsueh-Yung Lin.
\newblock The bimeromorphic {K}odaira problem for compact {K}{\"a}hler
  threefolds of {K}odaira dimension 1.
\newblock {\em arXiv preprint}, 1612.09271, 2016.

\bibitem[Lin17a]{Lin17a}
Hsueh-Yung Lin.
\newblock Algebraic approximations of compact {K}{\"a}hler threefolds of
  {K}odaira dimension 0 or 1.
\newblock {\em arXiv preprint}, 1704.08109, 2017.

\bibitem[Lin17b]{Lin17b}
Hsueh-Yung Lin.
\newblock Algebraic approximations of uniruled compact {K}{\"a}hler threefolds.
\newblock {\em arXiv preprint}, 1710.01083, 2017.

\bibitem[Nak88]{NakW}
Noboru Nakayama.
\newblock On {W}eierstrass models.
\newblock In {\em Algebraic geometry and commutative algebra, {V}ol.\ {II}},
  pages 405--431. Kinokuniya, Tokyo, 1988.

\bibitem[Nak02a]{Nak02c}
Noboru Nakayama.
\newblock Global structure of an elliptic fibration.
\newblock {\em Publ. Res. Inst. Math. Sci.}, 38(3):451--649, 2002.

\bibitem[Nak02b]{NakLocal}
Noboru Nakayama.
\newblock Local structure of an elliptic fibration.
\newblock In {\em Higher dimensional birational geometry ({K}yoto, 1997)},
  volume~35 of {\em Adv. Stud. Pure Math.}, pages 185--295. Math. Soc. Japan,
  Tokyo, 2002.

\bibitem[Nam01]{Nam01}
Yoshinori Namikawa.
\newblock Extension of 2-forms and symplectic varieties.
\newblock {\em J. Reine Angew. Math.}, 539:123--147, 2001.

\bibitem[Sch12]{Sch12}
Florian Schrack.
\newblock Algebraic approximation of {K}\"ahler threefolds.
\newblock {\em Math. Nachr.}, 285(11-12):1486--1499, 2012.

\bibitem[Ser58]{Ser58}
Jean-Pierre Serre.
\newblock Sur la topologie des vari\'et\'es alg\'ebriques en caract\'eristique
  {$p$}.
\newblock In {\em Symposium internacional de topolog\'\i a algebraica
  {I}nternational symposium on algebraic topology}, pages 24--53. Universidad
  Nacional Aut\'onoma de M\'exico and UNESCO, Mexico City, 1958.

\bibitem[Ser06]{Ser06}
Edoardo Sernesi.
\newblock {\em Deformations of algebraic schemes}, volume 334 of {\em
  Grundlehren der Mathematischen Wissenschaften}.
\newblock Springer-Verlag, Berlin, 2006.

\bibitem[Siu87]{Siu87}
Yum~Tong Siu.
\newblock Strong rigidity for {K}\"ahler manifolds and the construction of
  bounded holomorphic functions.
\newblock In {\em Discrete groups in geometry and analysis ({N}ew {H}aven,
  {C}onn., 1984)}, volume~67 of {\em Progr. Math.}, pages 124--151.
  Birkh\"auser Boston, Boston, MA, 1987.

\bibitem[Tak03]{Tak03}
Shigeharu Takayama.
\newblock Local simple connectedness of resolutions of log-terminal
  singularities.
\newblock {\em Internat. J. Math.}, 14(8):825--836, 2003.

\bibitem[Var89]{Var89}
Jean Varouchas.
\newblock K\"ahler spaces and proper open morphisms.
\newblock {\em Math. Ann.}, 283(1):13--52, 1989.

\bibitem[Voi02]{V02}
Claire Voisin.
\newblock {\em Th\'eorie de {H}odge et g\'eom\'etrie alg\'ebrique complexe},
  volume~10 of {\em Cours Sp\'ecialis\'es [Specialized Courses]}.
\newblock Soci\'et\'e Math\'ematique de France, Paris, 2002.

\bibitem[Voi04]{Voi04}
Claire Voisin.
\newblock On the homotopy types of compact {K}\"ahler and complex projective
  manifolds.
\newblock {\em Invent. Math.}, 157(2):329--343, 2004.

\bibitem[Voi06]{Voi07}
Claire Voisin.
\newblock On the homotopy types of {K}\"ahler manifolds and the birational
  {K}odaira problem.
\newblock {\em J. Differential Geom.}, 72(1):43--71, 2006.

\bibitem[W{\l}o09]{Wlo09}
Jaros{\l}aw W{\l}odarczyk.
\newblock Resolution of singularities of analytic spaces.
\newblock In {\em Proceedings of {G}\"okova {G}eometry-{T}opology {C}onference
  2008}, pages 31--63. G\"okova Geometry/Topology Conference (GGT), G\"okova,
  2009.

\bibitem[Zuc79]{Zuc79}
Steven Zucker.
\newblock Hodge theory with degenerating coefficients. {$L_{2}$}\ cohomology in
  the {P}oincar\'e metric.
\newblock {\em Ann. of Math. (2)}, 109(3):415--476, 1979.

\end{thebibliography}

\end{document}